\documentclass{article}
\usepackage[utf8]{inputenc}
\usepackage{graphicx}
\usepackage{amsmath}
\usepackage{amscd}

\usepackage[colorlinks=true,citecolor=red,linkcolor=blue,%
urlcolor=RubineRed,pdfpagetransition=Blinds,pdftoolbar=true,pdfmenubar=false]{hyperref}
\usepackage{color}

\usepackage{amssymb}
\usepackage{amsthm}
\usepackage{amsfonts}
\usepackage{eucal}
\usepackage{xspace}
\newtheorem{theorem}{Theorem}

\newtheorem{lemma}[theorem]{Lemma}

\newcommand{\R}{\mathbb{R}}
\newcommand{\Z}{\mathbb{Z}}
\newcommand{\C}{\mathbb{C}}
\newcommand{\N}{\mathbb{N}}

\begin{document}
\title{Numerical approximation of the scattering amplitude in elasticity}
\author{ \textsc{Juan A. Barceló$\,^\ast$}\and \textsc{Carlos Castro}\thanks{M2ASAI Universidad Polit\'ecnica de Madrid, Departamento de Matem\'atica e Inform\'atica, ETSI Caminos, Canales y Puertos, 28040 Madrid, Spain. E-mails: {\tt juanantonio.barcelo@upm.es,} {\tt carlos.castro@upm.es.} }
} 
 \date{}

\maketitle

\abstract{We propose a numerical method to approximate the scattering amplitudes for the elasticity system with a non-constant matrix potential in dimensions $d=2$ and $3$. This requires to approximate first the scattering field, for some incident waves, which can be written as the solution of a suitable Lippmann-Schwinger equation. In this work we adapt the method introduced by G. Vainikko in \cite{V} to solve such equations when considering the Lamé operator.  Convergence is proved for sufficiently smooth potentials. Implementation details and numerical examples are also given.}

\bigskip


\section{Introduction}

The scattering of time-harmonic elastic waves in $\mathbb{R}^d$ ($d=2,3$)  is described through the solutions of the Lamé equation
\begin{equation} 
\Delta^{\ast}\mathbf{u}(x)+\omega^{2}\mathbf{u}(x) = Q(x)\mathbf{u}%
(x), \qquad\qquad\omega>0,\ x\in\mathbb{R}^{d}, \label{ecuacionV}%
\end{equation}
where $\mathbf{u}: \mathbb{R}^{d} \to \mathbb{R}^{d} $ is the displacement vector,  $\omega >0$ is the frequency and
\begin{equation}
\label{operador}\Delta^{\ast}\mathbf{u}(x) = \mu\Delta\mathrm{I}%
\mathbf{u}(x)+(\lambda+\mu)\nabla div\,\mathbf{u}(x),
\end{equation}
with $\Delta\mathrm{I}$ denoting the $d\times d$ diagonal matrix with the Laplace operator on the diagonal. The constants $\lambda$ and $\mu$ are known as the Lam\'{e} parameters and depend on the underlying elastic properties. Throughout this paper we will assume that $\mu>0$ and $2\mu+\lambda>0$ so that the operator $\Delta^{\ast}$ is strongly elliptic.
The square matrix $Q$ represents the action of some live loads on a bounded region. This can also represent inhomogeneities in the density of the elastic material $\rho(x)$. In this case, $Q$ takes the particular form $Q=\omega^2 (1-\rho(x))I$. Here we assume that each component of $Q$, $q_{\alpha \beta}(x)$, $\alpha,\beta =1,...,d$, is real, compactly supported and belongs to $L^{r}(\mathbb{R}^{d})$ for some $r>d/2$. 

When $Q=0$ a solution of (\ref{ecuacionV}) is expressed as the sum of its 
 compressional part $\mathbf{u}^{p}$  and the shear part $\mathbf{u}^{s}$ 
\begin{equation}
\label{hdeco}\mathbf{u}=\mathbf{u}^{p}+\mathbf{u}^{s},
\end{equation}
where
\begin{equation}
\mathbf{u}^{p}=-\frac{1}{k_{p}^{2}}\nabla div\,\mathbf{u}\quad\text{and}
\quad\mathbf{u}^{s}=\mathbf{u}-\mathbf{u}^{p} \label{upus}%
\end{equation}
with
\begin{equation*}
\displaystyle k_{p}^{2}=\frac{\omega^{2}}{(2\mu+\lambda)}
\qquad\mathtt{\mathrm{and}} \qquad k_{s}^{2}=\frac{\omega^{2}}{\mu}.
\end{equation*}
They are solutions of the \emph{vectorial homogeneous Helmholtz equations}
$$
\Delta\mathrm{I}\mathbf{u}^{p}(x)+k_{p}^{2}\mathbf{u}^{p}(x)=\mathbf{0}, \qquad \mbox{and} \qquad 
\Delta\mathrm{I}\mathbf{u}^{s}(x)+k_{s}^{2}\mathbf{u}^{s}(x)=\mathbf{0},
$$ 
respectively. The values $k_{p}$ and $k_{s}$ are known as the associated speed of propagation of longitudinal waves
(\textit{p-waves}) and transverse waves (\textit{s-waves}) respectively. Note that compressional waves satisfy $\nabla \times \mathbf{u}^{p}=0$ while for the transverse ones we have $div \;  \mathbf{u}^{s}=0$.

When $Q\neq 0$ solutions $\mathbf{u}$ of \eqref{ecuacionV} can be interpreted as perturbations of the homogeneous ones, i.e. we write
\begin{equation}
\mathbf{u}(x)=\mathbf{u}_{i}(x)+\mathbf{v}(x), \label{uuiv}%
\end{equation}
where $\mathbf{u}_{i}$ is the incident wave, a solution of the homogeneous
Lamé equation (i.e. with $Q=0$), and $\mathbf{v}$ the
\emph{scattered solution}. 

Unicity of scattered solutions $\mathbf{v}$ is obtained from radiation conditions when $|x|\to \infty$. A natural choice is to assume that there are no reflections coming from infinity. As $Q$ is compactly supported this condition coincides with the analogous imposed to solutions of $\Delta^{\ast}\mathbf{v}(x)+\omega^{2}%
\mathbf{v}(x)=\mathbf{0}$ in an exterior domain. These are known as the outgoing
Kupradze radiation conditions and they are equivalent to make both $\mathbf{v}^{p}$ and $\mathbf{v}^{s}$ (according to the decomposition (\ref{hdeco})-(\ref{upus})) satisfy
the corresponding outgoing Sommerfeld radiation conditions for the Hemholtz equation, that is,
\begin{align}
(\partial_{r}-ik_{p})\mathbf{v}^{p}  &  =\mathbf{o}(r^{-(d-1)/2}),\qquad
r=|x|\rightarrow\infty,\label{radiacionup}\\
(\partial_{r}-ik_{s})\mathbf{v}^{s}  &  =\mathbf{o}(r^{-(d-1)/2}),\qquad
r=|x|\rightarrow\infty. \label{radiacionus}%
\end{align}

It is known that, with certain conditions in $Q$, the system \eqref{ecuacionV},\eqref{uuiv}
together with the outgoing Kupradze radiation conditions \eqref{radiacionup}-\eqref{radiacionus} has a unique solution $\mathbf{u}$  (see Theorem 2.1 and Proposition 3.1 of \cite{BFPRV2} and chapter 5 in \cite{H}).

Problem \eqref{ecuacionV}, \eqref{radiacionup}-\eqref{radiacionus} and \eqref{uuiv} is equivalent to the integral equation
\begin{equation} \label{eq_LS}
\mathbf{u}=\mathbf{u}_i + \int_{\mathbb{R}^d} \Phi (x-y) Q(y)\mathbf{u}(y) \; dy, 
\end{equation}
where $\Phi(x)$ is the fundamental tensor of the Lamé operator $\Delta^*+\omega^2 I$. This tensor is well-known for dimension $d=2,3$ (see \cite{A} and \cite{K}). We give its expression in section \ref{coeficientes} below.   

The implicit integral equation for $\mathbf{u}$ in \eqref{eq_LS} is known as a Lippmann-Schwinger equation. A trigonometrical collocation method was proposed in \cite{V} for the numerical approximation of this equation, when considering the scalar case and the fundamental solution of the Hemholtz equation, (we used this method in \cite{BCR} to give approximations of the potential in the inverse quantum scattering problem). A similar collocation method, but without trigonometric polynomials, was proposed in \cite{LMS}. In this paper we adapt the method in \cite{V} to approximate the solutions of \eqref{eq_LS}. Apart of the fact that we are dealing with a more complex vector problem, there are two main difficulties arising. The first one comes from the asymptotic bound for the Fourier coefficients of a suitable truncation of the Green tensor field associated to the Lamé operator (see Lemma \ref{le_coef} below). This is required to prove the convergence of the method. The second one comes from the fact that we slightly change the point of view in \cite{V} where the numerical method gives an approximation of the product $Q \mathbf{u}$, and therefore $\mathbf{u}$ can  be computed only in the support of $Q$. As described in \cite{V}, the solution $\mathbf{u}$ can be extended to $\mathbb{R}^d$ but the estimate of the numerical approximation blows up in the proximity of the boundary of the support of $Q$. In our approach we solve the problem directly on $\mathbf{u}$ and therefore we obtain a global approximation in $\mathbb{R}^d$. 

As an application we approximate scattering amplitudes. Let us briefly describe what these objects are and their interest. As incident waves we usually consider plane waves either transverse (plane \textit{s-waves})
\begin{equation} 
\label{uis}\mathbf{u}^{s}_{i}(\omega, \theta, \varphi,x)=e^{ik_{s}\theta\cdot x}\varphi,
\end{equation}
with polarization vector $\varphi\in\mathbb{S}^{d-1}$ orthogonal to the wave
direction $\theta\in\mathbb{S}^{d-1}$, or longitudinal plane waves (plane
\textit{p-waves})
\begin{equation}
\label{uip}\mathbf{u}^{p}_{i}(\omega,\theta,x)=e^{ik_{p}\theta\cdot x}\theta.
\end{equation}
Note that transversal waves depend on an extra angle parameter $\varphi$, not determined by $\theta$, only for $d=3$. From now on, we omit the dependence on $\varphi$ when considering the specific case $d=2$.
 
If $\mathbf{u}_p$ is the solution of (\ref{ecuacionV}) with $\mathbf{u}_p=\mathbf{u}_i^p+\mathbf{v}_p$,  $\mathbf{v}_p$ the scattered solution satisfying the outgoing Kupradze radiation condition, then $\mathbf{v}_p=\mathbf{v}^p_p+\mathbf{v}^s_p$ and we have  the asymptotic as $|x|\rightarrow\infty$ (see section 2 of \cite{BFPRV2} or \cite{BFPRV1}),
\begin{align}
\mathbf{v}^{p}_p(x)  &  =  c\ k_{p}^{\frac{d-3}{2}}\,\frac{e^{ik_{p}|x|}}%
{|x|^{(d-1)/2}}\,\mathbf{v}_{p,\infty}^p\left(\omega,\theta, x/|x|\right)  +
\mathbf{o}\left(  |x|^{-(d-1)/2}\right)  , \label{pp_asymptotic}\\
\mathbf{v}^{s}_p(x)  &  = c\ k_{s}^{\frac{d-3}{2}}\,\frac{e^{ik_{s}|x|}}%
{|x|^{(d-1)/2}}\,\mathbf{v}_{p,\infty}^s\left(\omega,\theta,  x/|x|\right)  +
\mathbf{o}\left(  |x|^{-(d-1)/2}\right)  , \label{ps_asymptotic}
\end{align}
where $\mathbf{v}_{p,\infty}^p$ and $\mathbf{v}_{p,\infty}^s$ are known as the
longitudinal  and transverse  scattering amplitudes of $\mathbf{u}_p$
respectively. Note that the last argument in $\mathbf{v}_{p,\infty}^p\left(\omega,\theta, x/|x|\right)$ corresponds to a direction in $\mathbb{S}^{d-1}$. This is interpreted as the direction in which the scattering amplitude is observed. 

In a similar way, if
 $\mathbf{u}_s$ is the solution of (\ref{ecuacionV}) with $\mathbf{u}_s=\mathbf{u}_i^s+\mathbf{v}_s$,   $\mathbf{v}_s$ the scattered solution satisfying the outgoing Kupradze radiation condition  then $\mathbf{v}_s=\mathbf{v}^p_s+\mathbf{v}^s_s$ and we obtain the similar asymptotic  to (\ref{pp_asymptotic}) and (\ref{ps_asymptotic})  and the
longitudinal  and transverse  scattering amplitudes of $\mathbf{u}_s$, 
\begin{equation} \label{ss_asymptotic}
\mathbf{v}_{s,\infty}^p(\omega,\theta,\varphi, x/|x|), \qquad \mathbf{v}_{s,\infty}^s(\omega,\theta,\varphi,x/|x|).
\end{equation}

Scattering amplitudes can be easily measured in practice with seismographs situated far away from the support of the potential $Q$. Thus, a natural question is if we can derive information of the elastic material from these scattering amplitudes, \cite{H1}, \cite{H2}, \cite{BFPRV2}, \cite{BFPRV1} and \cite{BCLV} for numerical results. This is known as the inverse scattering problem in elasticity, (analogous problems arise in quantum scattering, acoustic or electromagnetic obstacle scattering, \cite{CK}). In our case, the inverse problem would be to obtain information about the matrix $Q$. 

To understand the problem let us mention how much information we can derive from these data. We can define $2$ scattering amplitudes coming from longitudinal incident waves $\mathbf{u}^{p}_{i}$. They depend on the variables $\left(\omega,\theta,  x/|x|\right)\in \mathbb{R}^+\times \mathbb{S}^{d-1} \times \mathbb{S}^{d-1} $. This gives us $2d-1$ free parameters. As they are vector waves, each scattering amplitude provides $d(2d-1)$ parameters. Thus, we have $2d(2d-1)$ parameters that we can obtain from the longitudinal incident waves. Concerning transverse incident waves $\mathbf{u}^{s}_{i}$ we can define
$2(d-1)$ different scattering amplitudes. Each one depends now on the variables $\left(\omega,\theta,\varphi,  x/|x|\right)\in \mathbb{R}^+\times \mathbb{S}^{d-1} \times \mathbb{S}^{d-2} \times \mathbb{S}^{d-1} $. We have now $3d-3$ parameters for each vector wave, and therefore $d(3d-3)$ total parameters coming from each scattering amplitude. Thus, we have  $2(d-1)d(3d-3)$ values that we can obtain from the transverse incident waves. Considering both longitudinal and transverse incident waves we sum up $2d(2d-1)+2(d-1)d(3d-3)$ scattering amplitudes that we can measure theoretically. In the particular case $d=2$ we obtain $24$ different scattering amplitudes, and $102$ when $d=3$.  

Now, if we want to determine $Q$ from these scattering amplitudes we have to compute $d^2$ functions depending on $d$ coordinates. In dimension $d=2$ this requires information depending on $8$ variables, while for $d=3$,  $27$ variables are involved. We see that, in principle the inverse problem consisting in recovering the matrix $Q$ from scattering data is overdetermined. Of course, some of these scattering amplitudes may not be independent and, in practice,  we cannot measure all these parameters accurately but it is reasonable to think that we can combine part of this information to recover $Q$.  Then, it is natural to restrict the scattering data. In the linear elasticity problem, it is most usual to deal with either fixed-angle or backscattering data. In the first case we fix the direction of the incident waves $\theta$ (which is no longer a free parameter) and in the latter, only scattering amplitudes with the last entry fixed as  $x/|x|=-\theta$ are considered.  Whether the scattering amplitudes allow to recover the potential $Q$ in these cases is widely open. There are some results only for some particular situations as for instance if $Q(x)=q(x)I$, or  when $\lambda+\mu=0$, where the Lamé operator can be reduced to a system of two independent vector Hemholtz equations, \cite{BFPRV2}. 

When reconstruction is not known, or difficult to obtain, it is sometimes possible to define approximations of $Q$ from scattering amplitudes (see \cite{BFPRV1} for backscattering data and \cite{BFPRV2} for fixed angle data). Let us give an example in dimension $d=2$. Note that scattering amplitudes for transversal incident waves will not depend on the angle $\varphi$ in this case and we omit this variable in the following. We define the Fourier transform of the $2\times 2$ matrix $Q_b$ (that we denote $\widehat{Q_b}$)  from scattering amplitudes as follows. For  $\xi \in \mathbb{R}^2$ with 
$\xi=-2 \omega \theta$, $\omega >0$ and $\theta \in \mathbb{S}^1$, 
$$\widehat{Q_b}(\xi)e_i=(\theta \cdot e_i) \mathbf{v_{p,\infty}}( \omega , \theta,-\theta)+ (\theta^\perp \cdot e_i) \mathbf{v_{s,\infty}}( \omega , \theta,-\theta) , \hspace{0.3cm} i=1,2,$$
where $ \mathbf{v_{p,\infty}}$ (resp. $ \mathbf{v_{s,\infty}}$) is a linear combination of $ \mathbf{v^p_{p,\infty}}(\omega_{pp}, \theta,-\theta)$ and $ \mathbf{v^s_{p,\infty}}(\omega_{ps},\theta, -\theta)$,  for certain energies $\omega_{pp}$ and $\omega_{ps}$ (resp. $ \mathbf{v^p_{s,\infty}}(\omega_{pp}, -\theta)$ and $ \mathbf{v^s_{s,\infty}}(\omega_{ps}, -\theta)$),  and $\{e_1,e_2\}$ is the canonical basis in $\mathbb{R}^2$. 
The matrix $Q_b$ is known as the Born approximation of $Q$ for backscattering data. There are not estimates on how good is the Born approximation in general. Most of the results involve asymptotic estimates for their Fourier tranforms, that allows to deduce that the Born approximation $Q_b$ and $Q$ share the same singularities. In other words, we can use the Born approximation to recover the singularities of $Q$, \cite{BFPRV1}. 

It is worth mentioning that here we are restricting ourselves to the recovering of $Q$ from scattering data. However, similar problems can be stated for more general situations as to recover the Lam\'e coefficients or other quantities depending on the elastic properties, 	\cite{H1}. This is a field where there are few results.  

As we mentioned before, we apply the numerical approximation of the Lippmann-Scwinger equation \eqref{eq_LS} to simulate scattering amplitudes. In fact, as we show below (section 6) the scattering amplitudes can be recovered from the potential and the scattering solutions of \eqref{eq_LS} for the different incident waves. This allows to construct synthetic scattering amplitudes from numerical approximation of \eqref{eq_LS}. In particular, we can test reconstruction algorithms and compute Born approximations that can be compared with the potential $Q$. This will be done in a forthcoming publication. 

The rest of the paper is divided as follows: In section 2 we reduce the Lippmann-Schwinger equation \eqref{eq_LS} to a multiperiodic problem by localizing and periodic extension. In section 3 we introduce the finite trigonometric space used in the discretization. In section 4 we derive the trigonometric collocation method. Convergence is proved in section 5. In section 6 we show how to approximate the scattering amplitudes. In section 7 we estimate the Fourier coefficientes of the Green function involved in the multiperiodic version of \eqref{eq_LS}. Finally, in section 8 we include some numerical examples that illustrate the convergence of the method and the approximation of some scattering amplitudes.


\section{Reduction to a multiperiodic problem}

We first write the Lippmann-Schwinger equation (\ref{eq_LS}) in terms of the scattered solution $\textbf{v}$,
\begin{eqnarray} \nonumber 
\mathbf{v}&=& \int_{\mathbb{R}^d} \Phi (x-y) Q(y)\mathbf{u}_i(y) \; dy+\int_{\mathbb{R}^d} \Phi (x-y) Q(y)\mathbf{v}(y) \; dy \\ \label{eq_LSv}
&=& \mathbf{f} +\int_{\mathbb{R}^d} \Phi (x-y) Q(y)\mathbf{v}(y) \; dy, 
\end{eqnarray}
where we have written the first term in the right hand side as a generic smooth function $\mathbf{f}:\mathbb{R}^d\to \mathbb{R}^d$. 

From now on we assume that, for some $\rho>0$, the components of $Q$, $Q_{\alpha \beta}$, and those of $ \mathbf{f}=(f_1,f_2,...,f_d)^T$ satisfy,
\begin{equation}
supp \; Q_{\alpha \beta} \subset \overline{B}(0,\rho), \; \; Q_{\alpha \beta}\in W^{\nu,2}(\mathbb{R}^d), \; \; f_\alpha \in W^{\nu,2}_{loc}(\mathbb{R}^d), \; \; \nu>d/2.
\label{eq:reg_as}
\end{equation}
In particular this guarantees that the components of $Q$ and $\mathbf{f}$ are continuous functions. 

Given $R>2\rho$, we define
$$
G_R=\left\{ x=(x_1,x_2,...,x_d)\in \R^d \; : \; |x_k|<R,\; k=1,2,..,d \right\} .
$$

Problem (\ref{eq_LSv}) on $G_R$ is equivalent to the following: given $\mathbf{f}$ and the $d\times d$ matrix function $Q$, find $\mathbf{v}:G_{R}\to \mathbb{R}^d$ solution of 
\begin{equation} \label{eq_LSv2}
\mathbf{v}=\mathbf{f}+\int_{G_R} \Phi (x-y) Q(y)\mathbf{v}(y) \; dy ,\qquad (x\in G_R). 
\end{equation}
Solving \eqref{eq_LSv2} in $x\in G_R$ allows us to obtain $\mathbf{v}$ in $G_R$   . Once this is known, the solution in $\mathbb{R}^d\backslash G_R $   can be obtained by simple integration using (\ref{eq_LS}). This allows us to localize the problem to the bounded domain $G_R$.

The idea now is to use \eqref{eq_LSv2} to approximate the solution $\mathbf{v}$ in the smaller ball $x\in \overline{B}(0,\rho)\subset G_R$. For those points, only the values of $\Phi$ in the ball $\overline{B}(0,2\rho)$ are involved and therefore, changing $\Phi$ outside this ball does not affect to the solution $\mathbf{v}(x)$. This allows us to localize the Green tensor function to a compact support function in $G_R$ without changing the solution. We consider a smooth cutting of the form 
\begin{equation} \label{mat_K}
K(x)=\Phi(x) \psi(|x|), \quad x\in G_R, \quad R>2\rho,
\end{equation}
where $\psi:[0,\infty)\to \mathbb{R}$ satisfy the conditions
$$
\psi \in C^\infty[0,\infty), \quad  \psi(r)=1 \mbox{ for $0\leq r \leq 2\rho$,} \quad  \psi(r)=0 \mbox{ for $r\geq R$.}
$$

Analogously, we truncate $\mathbf{f}$ by considering $\mathbf{f}(x)\psi (|x|)$, that we still denote $\mathbf{f}$ to simplify.

A more drastic cutting using the characteristic function of the  ball $B(0,R)$ is also possible but this is more convenient  as we comment below. 

Once localized the problem we extend $\mathbf{v},$ $\mathbf{f}$, the components of the matrixes $Q$ and $K$ to $R$-periodic functions for which we still use same notation. Thus, we change \eqref{eq_LSv2} by a multiperiodic integral equation 
\begin{equation} \label{eq_LSvp}
\mathbf{v}=\mathbf{f}+\int_{G_R} K (x-y) Q(y)\mathbf{v}(y) \; dy ,\qquad (x\in G_R). 
\end{equation}
The smooth cutting of $\mathbf{f}$ and $\Phi(x)$ guarantees a smooth periodic extension from $G_R$ to $\mathbb{R}^d$. 

The solvability of \eqref{eq_LSvp} is obviously deduced from the one of \eqref{eq_LSv2}. Moreover, once solved 
\eqref{eq_LSvp} we obtain the solution of \eqref{eq_LSv2} in $B(0,\rho)$ and we can extend $ \mathbf{v}$ to $\mathbb{R}^d$ by 
\begin{equation} \label{eq_sol_out}
\mathbf{v}=\mathbf{f}+\int_{B(0,\rho)} \Phi (x-y) Q(y)\mathbf{v}(y) \; dy ,\qquad x\in \mathbb{R}^d, 
\end{equation}
where  $\mathbf{f}$ in (\ref{eq_sol_out}) represents the original function, i.e. before cutting and periodizing. 

\section{Finite dimensional trigonometric space}

The family of exponentials
\begin{equation} \label{eq_fij}
\varphi_j(x)=\frac{e^{i\pi j \cdot x/R}}{(2R)^{d/2}}, \quad  j=(j_1,j_2,...,j_d)\in\Z^d,
\end{equation}
constitutes an orthonormal basis on $L^2(G_R)$ with the norm
$$\| u \|_0^2=\int_{G_R}|u(x)|^2 dx.$$ We also introduce the space $H^\eta=H^\eta(G_R)$ which consists of $dR-$multiperiodic functions (distributions) having finite norm
$$
\| u \|_\eta=\left( \sum_{j\in\Z^d}(1+ |j|)^{2\eta} |\hat u(j)|^2 \right)^{1/2} ,
$$
with
$$
\hat u(j)=  \int_{G_R} u(x)\overline{\varphi_j(x)} dx, \quad j\in \Z^d,
$$
the Fourier coefficients of $u$.

We now introduce a finite dimensional approximation of $H^\eta$. Let us consider $h=2R/N$ with $N\in \N$ and a mesh on $G_R$ with grid points $jh$, $j\in \Z_h^d$ and
$$
\Z_h^d=\left\{ j=(j_1,j_2,...,j_d)\in \mathbb{Z}^d \; : \; -\frac{N}{2} \leq j_k < \frac{N}{2}, \; k=1,2,...,d \right\}.
$$
We also consider $\mathcal{T}_h$ the finite dimensional subspace of trigonometric polynomials of the form
$$
v_h=\sum_{j\in Z_h^d} c_j \varphi_j, \qquad c_j\in \C.
$$
Any $v_h\in \mathcal{T}_h$ can be represented either through the Fourier coefficients
or the nodal values,
$$
v_h(x)=\sum_{j\in Z_h^d} \hat v_h(j) \; \varphi_j(x)=\sum_{j\in Z_h^d} v_h(jh) \; \varphi_{h,j}(x),
$$
where $\varphi_{h,j}(kh)=\delta_{jk}$, more specifically
$$
\varphi_{h,j}(x)=\frac{h^d}{(2R)^d}\sum_{k\in \Z_h^d} e^{i\pi k \cdot (x-jh)/R}.
$$

For a given $v_h \in \mathcal{T}_h$, the nodal values $\bar v_h$ and the Fourier coefficients $\hat v_h$ are related by the discrete Fourier transform $\mathcal{F}_h$ as follows,
$$
\hat v_h =h^d\mathcal{F}_h \bar v_h, \qquad \bar v_h =\frac{1}{h^d}\mathcal{F}_h^{-1} \hat v_h,
$$
where, as usual, $\mathcal{F}_h$ relates the sequence $x(n)$ ($n=(n_1,n_2,...,n_d)$) with $X(j)$ according to
\begin{eqnarray*}
X(j)&=&\sum_{n_1,n_2,...,n_d=-N/2}^{N/2-1} x(n)e^{-i2\pi n\cdot j /N}, \qquad j=(j_1,j_2, ...,j_d), \\
&& j_k=-N/2,-N/2+1,...,N/2-1 .
\end{eqnarray*}
This definition coincides with the usual one in numerical codes (as MATLAB) up to a translation, since it considers instead $j_k=0, ..., N-1$. This must be taken into account in the implementation. 

The orthogonal projection from $H^\eta$ to $\mathcal{T}_h$ is defined by the formula
$$
P_h v=\sum_{j\in Z_h^d} \hat v(j) \varphi_j,
$$
while the interpolation projection $S_hv$ is defined, when $\eta>d/2$, by
$$
S_hv\in \mathcal{T}_h, \quad (S_hv)(jh)=v(jh), \quad j\in \Z^d_h.
$$
(We need that  $v$ is continuous in order to be able to define $S_hv$. Since $H^\eta \subset   W^{\eta, 2}_{loc}(\mathbb{R}^d)$, a function $v \in H^\eta$ with $\eta >d/2$ is continuous).

For a vector valued function $\mathbf{v}= (v^1, v^2,...,v^d)^T$ we maintain the same notation, i.e.   
$$
P_h \mathbf{v} = \left( P_h(v^1), P_h(v^2),...,P_h(v^d)	\right)^T \in \mathcal{T}_h^d, 
$$
$$
S_h \mathbf{v} = \left( S_h(v^1), S_h(v^2),...,S_h(v^d)	\right)^T \in \mathcal{T}_h^d,
$$
and analogously for the discrete Fourier transform $\mathcal{F}_h$ of a vector $\mathbf{v_h} \in (\mathcal{T}_h)^d$.  $\mathbf{v} \in  (H^\eta)^d$ if 
$$\|\mathbf{v}\|_\eta=\max_{k=1,2, ...,d }\|v^k\|_\eta < \infty.$$

\section{Trigonometric collocation method}

The trigonometric collocation method to solve \eqref{eq_LSvp} reads, 
\begin{equation} \label{eq_LSv3hp}
\mathbf{v_h}=\mathbf{f_h}+S_h(\mathcal{K}(Q\mathbf{v_h})), \quad \mathbf{v_h} \in (\mathcal{T}_h)^d,
\end{equation}
where $\mathbf{f_h}=S_h \mathbf{f}$ and
\begin{equation}
\mathcal{K}(\mathbf{w_h})(x)=\int_{G_R} K (x-y) \mathbf{w_h}(y) \; dy, \quad \mathbf{w_h} \in (\mathcal{T}_h)^d.
\label{eq:kernelh}
\end{equation}
Note that here $K$ and $Q$ are $d\times d$ matrixes while $\mathbf{w_h}$ is a column vector with each component $ w^i_h \in \mathcal{T}_h$. The integral applies to each component of the integrand vector function. 

From the numerical point of view there is a difficulty in the discrete 
formulation  \eqref{eq_LSv3hp} since $Q\mathbf{v_h}\notin (\mathcal{T}_h)^d$ and should be approximated. Therefore we modify the discrete formulation as follows, 
\begin{equation} \label{eq_LSv3h}
\mathbf{v_h}=\mathbf{f_h}+\mathcal{K}(S_h(Q\mathbf{v_h})), \quad \mathbf{v_h} \in (\mathcal{T}_h)^d.
\end{equation} 
The operator $\mathcal{K}$ leaves invariant the subspace $(\mathcal{T}_h)^d$ and \eqref{eq_LSv3h} is consistent. 

Since each one of the components in the matrix $K$ is a periodic function the eigenfunctions of the associated convolution operator are known to be $\varphi_j$, defined in \eqref{eq_fij}, while their eigenvalues $\widehat K_{\alpha\beta}(j)$ are the Fourier coefficients, i.e.
\begin{equation} \label{eq_foco}
\int_{G_R} K_{\alpha \beta} (x-y) \; \varphi_j (y) \; dy = \widehat K_{\alpha\beta}(j) \varphi_j (x), \quad j\in \mathbb{Z}^d.
\end{equation}

We can write \eqref{eq_LSv3h} in matrix form too. Let us interpret $\mathbf{v_h}=(v_h^1,v_h^2,...,v_h^d)^T$ and $\mathbf{f_h}=(f_h^1,f_h^2,...,f_h^d)^T$ as $dN^d$-column vectors with the node values of the $d$ components respectively. We also write the matrix form of the discrete Fourier transform as $\mathcal{F}_h$ and observe that $S_h(Q\mathbf{v_h})$ can be interpreted as the vector obtained by the product of the components of the node values of $Q$. We obtain the matrix formulation, 
\begin{equation} \label{eq_LSv3m}
\mathbf{v_h}=\mathbf{f_h}+\mathcal{F}_h^{-1}\widehat  K_h \mathcal{F}_h \; (Q_h)_{diag} \mathbf{v_h}, 
\end{equation}
where $(Q_h)_{diag}$ is the $dN^d\times dN^d$ matrix whose components are the $N^d\times N^d$ matrixes $ diag(S_h(Q_{\alpha \beta})),$  $\alpha,\beta=1,...,d,$
contain the nodal values of $S_h(Q_{\alpha \beta})$ in the diagonal. Here $\widehat  K_h$ is also a $dN^d \times dN^d$ square matrix with components $ K_{\alpha \beta} $, $ \alpha,\beta=1,...,d,$ where each $\hat K_{\alpha \beta}$ is a diagonal $N^d\times N^d$ matrix with the Fourier coefficients of $K_{\alpha \beta}$. 

Thus, the matrix form of the collocation method at the nodes reads,
\begin{equation}
A_h{\bf v_h} = {\bf f_h}, \quad A_h=I-\mathcal{F}_h^{-1}\widehat K_h \mathcal{F}_h \; (S_h)_{diag},
\label{eq:col_met_nodes}
\end{equation}
and at the Fourier coefficients,
\begin{equation}
\widehat A_h{\bf \widehat v_h} = {\bf \widehat g_h}, \quad \widehat A_h=I-\widehat  K \mathcal{F}_h \; (S_h)_{diag}\mathcal{F}_h^{-1}, \quad {\bf \widehat g_h}=\mathcal{F}_h {\bf  f_h} .
\label{eq:col_met_fou}
\end{equation}

\section{Convergence of the trigonometric collocation method}

In this section we prove a convergence result for the collocation method described above. The proof requires three preparatory lemmas. 
\begin{lemma} \label{lema_vainikkko}
 Let $\eta > d/2$  be.
\begin{enumerate}
\item If $u,v \in H^\eta$ then $uv \in H^\eta $ and  $\|uv \|_\eta \leq c_\eta \| u\|_\eta \| u\|_\eta$.
\item $\|\mathbf{v}-S_h \mathbf{v}\|_\nu \leq c_{\nu , \eta, R} h^{\eta- \nu}\| \mathbf{v} \|_\eta \quad \textrm{for } 0 \leq \nu \leq \eta, \;\; \mathbf{v} \in (H^\eta)^d$.
\end{enumerate}

\end{lemma}
The proof of this Lemma can be found in \cite{V} or \cite{SV} (dimension two) and \cite{KV}.

\begin{lemma} \label{le_coef}
The Fourier coefficients of the components of the matrix $K$ defined in \eqref{mat_K} satisfy the following estimate
\begin{equation} \label{eq_le_coef}
|\widehat K_{\alpha \beta}(j)| \leq c |j|^{-2} \log |j|,  \quad j\in \mathbb{Z}^d, \quad  |j|\neq 0, \quad d=2,3,
\end{equation} 
for some constant $c>0$ that depends only on $\omega$, $\lambda$, $\mu$,   $R$, $\rho$ and $d$. 
\end{lemma}

We prove this lemma in section \ref{coeficientes} below.

\begin{lemma}
Assume that all the components in $Q$ satisfy $Q_{\alpha \beta}\in H^\eta$, with $\eta > d/2$  and let $\varepsilon > 0$ be such that    $\eta \geq 2-\varepsilon $ Then, 
\begin{equation} \label{eq_Kh1}
\| \mathcal{K}(Q \cdot) - \mathcal{K}(S_h(Q \cdot)) \|_{\mathcal{L}((H^\eta)^d,(H^\eta)^d)} \leq c_{\eta, \varepsilon ,R,Q}h^{2-\varepsilon},
\end{equation}
where $\mathcal{L}((H^\eta)^d,(H^\eta)^d)$ denotes the space of linear and continuous functions in $(H^\eta)^d$
(of course the constant $c_{\eta, \varepsilon ,R,Q}$  also depends on the constant of the Lemma \ref{le_coef}) .

\end{lemma}

\begin{proof}
As a consequence of Lemma \ref{le_coef} we have that $\mathcal{K} \in \mathcal{L} ((H^{\eta-(2-\varepsilon)})^d,(H^{\eta })^d)$ for any $\eta \in \mathbb{R}$ and $\varepsilon >0$. Assume that $\mathbf{v} \in (H^\eta)^\eta $.    By Lemma \ref{lema_vainikkko},
$$\| \mathcal{K}(Q  \mathbf{v}  ) - \mathcal{K}(S_h(Q  \mathbf{v} )) \|_\eta \leq c\|Q \mathbf{v}-S_hQ \mathbf{v}  \|_{\eta -(2-\varepsilon)} \leq   c_{\eta, \varepsilon ,R}h^{2-\varepsilon}  \| Q \mathbf{v}\|_\eta $$
$$ \leq c_{\eta, \varepsilon ,R}  \max_{\alpha , \beta}{\|Q_{\alpha \beta}} \|_\eta   h^{2-\varepsilon}      \| \mathbf{v} \|_\eta  \leq c_{\eta, \varepsilon ,R,Q}h^{2-\varepsilon}      \| \mathbf{v} \|_\eta. $$

\end{proof}

We have the following result
\begin{theorem} \label{th_conv}
Assume that $Q$ and $\mathbf{f}$ satisfy \eqref{eq:reg_as} and the homogeneous problem \eqref{eq_LS} with $\mathbf{u}_i=0$ has only the trivial solution. Then, equation \eqref{eq_LSvp} has a unique solution $\mathbf{v} \in (H^\nu)^d$, $\eta > d/2$, and the collocation equation \eqref{eq_LSv3h} has also a unique solution $\mathbf{v_h} \in (\mathcal{T}_h)^d$ for sufficiently small $h$. Moreover, 
\begin{equation}
\| \mathbf{v_h} -\mathbf{v} \|_\eta \leq c_{\eta, \nu, R, Q} \| \mathbf{v} - S_h \mathbf{v} \|_\eta \leq c_{\eta , \nu , R, Q} \|\mathbf{v} \|_\nu  h^{\nu -\eta}, \; \;  \nu \geq \eta. 
\label{eq:error}
\end{equation} 
\end{theorem}

\begin{proof}
This is a generalization of the analogous proof for the scalar case in \cite{V}. From Lemmas \ref{lema_vainikkko} and \ref{le_coef} we have $\mathcal{K}(Q \cdot) \in \mathcal{L}((H^\eta)^d,(H^\eta)^d)$ since
$$\|\mathcal{K} Q \mathbf{v}\|_\eta \leq c \| Q \mathbf{v}\|_{\eta}  \leq c_\eta \max_{\alpha , \beta}\|Q_{\alpha \beta}\|_\eta  \| \mathbf{v}\|_{\eta}  \leq c_{\eta , Q} \|  \mathbf{v}\|_{\eta}.$$

The existence and uniqueness for \eqref{eq_LSvp} is easily deduced from the inverse of the operator $I-\mathcal{K}(Q \cdot) \in \mathcal{L}((H^\eta)^d,(H^\eta)^d)$. In fact this inverse exists since $\mathcal{K}(Q \cdot)$ is compact and the homogeneous integral equation \eqref{eq_LSvp} with $\mathbf{f}=0$ has only the trivial solution, since if there were a non-zero solution $\mathbf{v}$, it would be a solution of (\ref{eq_LS}) with $u_i=0$ in $B(0,\rho)$ and extend to $\mathbb{R}^d$. 

Concerning the collocation equation \eqref{eq_LSv3h}, we write
$$I-\mathcal{K}(S_hQ \cdot)=I-\mathcal{K}(Q \cdot )+\mathcal{K}(Q \cdot )-\mathcal{K}(S_hQ \cdot).$$
 $I-\mathcal{K}(Q \cdot)$   admits an inverse in $ \mathcal{L}((H^\eta)^d,(H^\eta)^d)$. On the other hand  by (\ref{eq_Kh1}) we have
$$  \|\mathcal{K}(Q \cdot)  -\mathcal{K}(S_h Q \cdot) \|_{\mathcal{L}(H^\eta)^d, (H^\eta)^d)}  \leq c_{\eta, \varepsilon, R, Q} h^{2-\varepsilon}, \;\; \eta >d/2 \; \textrm{ and } \eta \geq  2-\varepsilon,$$
then if we take $h$ such that 
\begin{eqnarray*}
\|\mathcal{K}(Q \cdot)  -\mathcal{K}( S_hQ \cdot) \|_{\mathcal{L}(H^\eta)^d, (H^\eta)^d)} &\leq& c_{\eta, \varepsilon, R, Q} h^{2-\varepsilon} \\
 & <& \frac{1}{\| (I-\mathcal{K}(Q\cdot ))^{-1}  \|_{\mathcal{L}((H^\eta)^d, (H^\eta)^d)}} , 
\end{eqnarray*}
then we can guarantee the existence of the inverse of the operator $I-\mathcal{K}(S_hQ \cdot)$ in $\mathcal{L}((H^\eta)^d, (H^\eta)^d)$, (see Corollary 1.1.2 of \cite{SV}), and therefore the existence and uniqueness of solution $\mathbf{v}$ of 
\eqref{eq_LSv3h} in $(H^\eta)^d$. But if $\mathbf{v}$   is a solution of \eqref{eq_LSv3h}, then $\mathbf{v}$  is necessarily in $(\mathcal{T}_h)^d$.

We now obtain \eqref{eq:error}. First of all observe that from \eqref{eq_LSvp} and \eqref{eq_LSv3h} we obtain 
\begin{eqnarray*} \notag
&& [I-\mathcal{K} (S_h Q \;  \cdot)] ({\bf v}- {\bf  v_h})=  {\bf v} - \mathcal{K} (S_h Q \;  {\bf v}) - {\bf f_h} + S_h \mathcal{K} ( Q \;  {\bf v}) \\ && - S_h\mathcal{K} ( Q \;  {\bf v})  = {\bf v}-S_h{\bf v} +  S_h \mathcal{K} ( Q \;  {\bf v}) - \mathcal{K} (S_h Q \;  {\bf v}).
\label{eq_er1}
\end{eqnarray*}
As $[I-\mathcal{K} (S_h Q \;  \cdot)] $ is invertible in $\mathcal{L}((H^\eta)^d,(H^\eta)^d)$, for sufficiently small $h$, it is enough to estimate the second hand term in this last expression. For the first two terms, by Lemma \ref{lema_vainikkko}, we have, 
\begin{equation}\label{primer_sumando}  
\| {\bf v}-S_h{\bf v} \|_\eta \leq c_{\eta , \mu,R}h^{\nu-\eta} \| {\bf v} \|_\nu, \quad  \eta \leq \nu.  
\end{equation}
Now we estimate $S_h \mathcal{K} -  \mathcal{K}S_h $ as a linear operator in $(H^\eta)^d$. Note that the projector operator $P_h$ commutes with $\mathcal{K}$. To simplify we prove it only for the case $d=2$. For ${\bf w}= \sum_{j\in \mathbb{Z}^2}(w^1(j),w^2(j))^T \varphi_j$, we have
\begin{eqnarray*}
\mathcal{K} P_h {\bf w} &=& \mathcal{K} \sum_{j\in \mathbb{Z}_h^d} \left(  \begin{array}{c} \widehat w^1(j)   \\ \widehat w^2(j)  \end{array} \right)\varphi_j  \\ &=& \sum_{j\in \mathbb{Z}_h^d} \left( \begin{array}{cc}
\widehat K_{11}(j) & \widehat K_{12}(j) \\
\widehat K_{21} (j) & \widehat K_{22} (j)
\end{array} \right)\left(  \begin{array}{c} \widehat w^1(j)   \\ \widehat w^2(j)  \end{array} \right)\varphi_j \\
&=& P_h \sum_{j\in \mathbb{Z}^d} \left( \begin{array}{cc}
\widehat K_{11}(j) & \widehat K_{12}(j) \\
\widehat K_{21} (j) & \widehat K_{22} (j)
\end{array} \right)\left(  \begin{array}{c} \widehat w^1(j)   \\ \widehat w^2(j)  \end{array} \right)\varphi_j =  P_h \mathcal{K} {\bf w} .
\end{eqnarray*}
In particular, $(S_h \mathcal{K}-\mathcal{K}   S_h)=0$ on $(\mathcal{T}_h)^d$. Therefore, 
\begin{eqnarray*}
&& \| (S_h \mathcal{K}-\mathcal{K}   S_h) ( Q \;  {\bf v})\|_\eta \leq   \| (S_h \mathcal{K}-\mathcal{K} S_h) ( P_h(Q \;  {\bf v}))\|_\eta  \\
&& +  \| (S_h \mathcal{K}-\mathcal{K} S_h) ( (I-P_h)(Q \;  {\bf v}))\|_\eta \\ && =
\| (S_h \mathcal{K}-\mathcal{K} S_h) ( (I-P_h)(Q \;  {\bf v}))\|_\eta\leq \| S_h \mathcal{K} ( (I-P_h)(Q \;  {\bf v}))\|_\eta \\ && + \|  \mathcal{K} S_h ( (I-P_h)(Q \;  {\bf v}))\|_\eta . 
\end{eqnarray*}
The two terms here are estimated in a similar way. For the first one,  let $\varepsilon $ be   such that $0 <\varepsilon <2$,
\begin{eqnarray*}
&& \| S_h \mathcal{K} ( (I-P_h)(Q \;  {\bf v}))\|_\eta \leq \| \mathcal{K} ( (I-P_h)(Q \;  {\bf v}))\|_\eta \\ &&  \leq c_R \|  (I-P_h)(Q \;  {\bf v})\|_{\eta-(2-\varepsilon)} 
 \leq c_{\eta, \nu,R} h^{2- \varepsilon+\nu-\eta} \|  Q \;  {\bf v}\|_\nu \\ &&   \leq c_{\eta, \nu , R,Q} h^{2-\varepsilon+\nu-\eta}   \| {\bf v}\|_\nu .
\end{eqnarray*}
From the  above inequality  and (\ref{primer_sumando}) we obtain (\ref{eq:error}).
\end{proof}

The trigonometric collocation provides a numerical method to approximate the solution of \eqref{eq_LSvp}, which coincides with the solution of \eqref{eq_LS} in $x\in B(0,\rho)$. Outside this ball we can approximate the solution of  \eqref{eq_LS} by a suitable discretization of formula \eqref{eq_sol_out}. For example, with the trapezoidal rule we obtain,  
\begin{equation} \label{eq_uout}
\mathbf{v}(x)=\mathbf{f}(x)+h^d\sum_{j \in \mathbb{Z}^d_h} \Phi (x-jh) Q(jh)\mathbf{v_h}(jh) \;  ,\qquad |x|>B(0,\rho). 
\end{equation}
The error can be easily estimated from Theorem \ref{th_conv}, 
$$
| \mathbf{v_h}(x) - \mathbf{v}(x) | \leq  c(x) h^{\nu} \| \mathbf{v} \|_\nu , \quad \nu>2,
$$
where $c(x)=c|x|^{-(d-1)/2}$, $c$ independent of $x$ and $h$ and $|x|^{-(d-1)/2}$
 comes from the fundamental tensor of Lamé operator (see Section \ref{coeficientes}) and the asymptotic behaviour of $H^{(1)}_\nu (r)$ when $r\rightarrow \infty$.
 
Note that $c(x)$ does not blow up as $|x|\to \rho$, even if $\Phi (x-jh)$ becomes singular for some values of $j$. This is due to the fact that $B(0,\rho)$ has compact support included in $Q$ and therefore the terms $\Phi (x-jh) Q(jh)$ in (\ref{eq_uout}) vanish for those $j$. This is in contrast with the method in \cite{V}, where the error bound is lost for   $|x|\to \rho$.

\section{Numerical approximation of scattering data}

As we said in the introduction we usually consider incident waves as plane waves either transverse (plane \textit{s-waves}) or longitudinal (plane \textit{p-waves}).
For an incident p-wave $\mathbf{u}_i^p$ in the form (\ref{uip}) we can define two different scattering amplitudes given in (\ref{pp_asymptotic})-(\ref{ps_asymptotic}). These can be written as (see \cite{BFPRV2} or \cite{BFPRV1})
\begin{align}
\mathbf{v}_{p,\infty}^p\left(\omega, \theta ,  x/|x|\right)   &  =\frac{1}{2\mu+\lambda
}\,\Pi_{x/|x|}\widehat{Q\mathbf{u}_p}\left(  k_{p}x/|x|\right)
,\label{tamplitud3}\\
\mathbf{v}_{p,\infty}^s\left(\omega, \theta ,  x/|x|\right)   &  =\frac{1}{\mu
}\,(\mathrm{I}-\Pi_{x/|x|})\widehat{Q\mathbf{u}_p}\left(  k_{s}x/|x|\right)  ,
\label{tamplitud3b}%
\end{align}
where $\mathbf{u}_p$ is the solution of (\ref{eq_LS}) with $\mathbf{u}_i=\mathbf{u}_i^p$. 
Here  $\Pi_{x/|x|}  $ is the orthogonal projection onto the line defined by the vector $x$. 

Analogously, for an incident s-wave in the form (\ref{uis}) we can define two different scattering amplitudes given in (\ref{ss_asymptotic}). These can be written as
\begin{align}
\mathbf{v}_{s,\infty}^p\left(\omega, \theta,\varphi, x/|x|\right)   &  =\frac{1}{2\mu+\lambda
}\,\Pi_{x/|x|}\widehat{Q\mathbf{u}_s}\left(  k_{p}x/|x|\right)
,\label{tamplitud31}\\
\mathbf{v}_{s,\infty}^s\left(\omega, \theta,\varphi, x/|x|\right)   &  =\frac{1}{\mu
}\,(\mathrm{I}-\Pi_{x/|x|})\widehat{Q\mathbf{u}_s}\left(  k_{s}x/|x|\right)  ,
\label{tamplitud31b}%
\end{align}
where now $\mathbf{u}_s$ is the solution of (\ref{eq_LS}) with $\mathbf{u}_i=\mathbf{u}_i^s$.

Natural approximations of these scattering amplitudes are given by,
\begin{eqnarray*} 
\mathbf{v}_{p,\infty, h}^p( \omega,\theta,\theta' ) &=&  \frac{1}{2\mu+\lambda
}\, h^d \sum_{j\in \mathbb{Z}_h^d} e^{-i k_p\theta' \cdot j h}\left(Q\mathbf{u}_{p,h}( jh) \cdot \theta'\right) \theta',  \\
\mathbf{v}_{p,\infty,h}^s( \omega,\theta, \theta')&=& \frac{1}{\mu
}\, h^d \sum_{j\in \mathbb{Z}_h^d} e^{-i k_s\theta' \cdot j h}\left[  
Q\mathbf{u}_{p,h}( jh) - \left(Q\mathbf{u}_{p,h}( jh) \cdot \theta'\right) \theta' \right] , \\
\mathbf{v}_{s,\infty, h}^p( \omega,\theta,\varphi,\theta' ) &=&  \frac{1}{2\mu+\lambda
}\, h^d \sum_{j\in \mathbb{Z}_h^d} e^{-i k_p\theta' \cdot j h}\left(Q\mathbf{u}_{s,h}( jh) \cdot \theta'\right) \theta',  \\
\mathbf{v}_{s,\infty,h}^s( \omega,\theta,\varphi, \theta')&=& \frac{1}{\mu
}\, h^d \sum_{j\in \mathbb{Z}_h^d} e^{-i k_s\theta' \cdot j h}\left[  
Q\mathbf{u}_{s,h}( jh) - \left(Q\mathbf{u}_{s,h}( jh) \cdot \theta'\right) \theta' \right] ,
\end{eqnarray*}
where $\omega>0$ and $\theta,\theta'\in \mathbb{S}^{d-1}$, $\varphi\in \mathbb{S}^{d-2}$ and 
$$ 
Q\mathbf{u}_{r,h}( jh) =(S_h (Q\mathbf{u}_r))( jh)=Q(jh)\mathbf{u}_r(jh) , \hspace{0.3cm} \mathbf{r}=\mathbf{p} \textrm{ or } \mathbf{s}.    
$$

\section{Fourier coefficients of the Green function} \label{coeficientes}

In this section we prove Lemma \ref{le_coef}. We divide this section in two subsections where we consider separately the cases $d=2$ and $d=3$.

\subsection{The case $d=2$}
The fundamental solution when $d=2$ is given by (see \cite{K}),
\begin{equation*} \label{eq_fund_sol}
\Phi(x)=\Phi_1(|x|) I+\Phi_2(|x|) J(x), 
\end{equation*}
where for $x\in \mathbb{R}^2\backslash \{0\}$ (identified with a matrix $1 \times 2$) the matrix $J$ is given by 
\begin{equation*}
J(x)=\frac{x^Tx}{|x|^2},
\label{eq:Jw}
\end{equation*}
and for each $v>0$, the functions $\Phi_1$ and $\Phi_2$ are given by
\begin{eqnarray*} 
\Phi_1(v)&=&\frac{i}{4\mu} H^{(1)}_0 (k_s v)-\frac{i}{4\omega^2v}\left[ k_s H^{(1)}_1 (k_s v)- k_p H^{(1)}_1 (k_p v) \right],\\
\Phi_2(v)&=&\frac{i}{4\omega^2}\left[ \frac{2k_s}{v} H^{(1)}_1 (k_s v)- k_s^2H^{(1)}_0 (k_s v) \right. \\
&& \quad \left. - \frac{2k_p}{v} H^{(1)}_1 (k_p v)+ k_p^2H^{(1)}_0 (k_p v) \right],
\end{eqnarray*}
with $H^{(1)}_k$ the Hankel function of first kind and order $k$. For $v\to 0$, $\Phi_1 \sim - (k_s^2+k_p^2)/(4\pi\omega^2) \log v$, while $\Phi_2 \sim (k_p^4-k_s^4)/(16\pi\omega^2)v^2\log(v)$, so that the integral in (\ref{eq_LS}) is weakly singular.
Thus, we can write 
\begin{equation} \label{eq_c3}
\Phi_n(v)=\frac{1}{\pi} \log v \Psi_n(v)+ \Upsilon_n (v), \hspace{0.4cm} v >0,  \; \; n=1,2, 
\end{equation}
where $\Psi_n,   \Upsilon_n \in  C^\infty [0,\infty )$, $n=1,2$, are even functions. 

On the other hand, we remind that 
\begin{equation}\label{eq_c3_1}
K_{\alpha \beta}(x)= \left( \Phi_1(|x|)\delta_{\alpha \beta} + \Phi_2(|x|)\frac{x_\alpha x_\beta}{|x|^2} \right) \psi (|x|), \quad \alpha,\beta=1,2.
\end{equation}
Let $a>0 $ be a number to be determined later, $j \in \mathbb{Z}^2$ and we assume that $|j|^{-a } \leq \min\{1/2, \rho\} $. From now on $C$ indicates a positive and universal constant which depends only on   $\omega$,  $\Phi_n$ $n=1,2$, $\lambda$, $\mu$, $\rho$, $R$, $d$ and $a$.

We can write
\begin{equation}\label{descomposicion}
\widehat{K}_{\alpha \beta}(j)=\int_{B(0,|j|^{-a})} K_{\alpha \beta}(x)\varphi_{-j}(x)dx+\int_{|j|^{-a} \leq |x| \leq R} K_{\alpha \beta}(x)\varphi_{-j}(x)dx.
\end{equation}

We start by estimating the first integral in (\ref{descomposicion}), essentially by using the size of the ball $B(0,|j|^{-a})$. From (\ref{eq_c3}) and (\ref{eq_c3_1})
\begin{equation*}\label{primera}
 \left| \int_{B(0,|j|^{-a})} K_{\alpha \beta}(x)\varphi_{-j}(x)dx \right|  \leq  C \int_{B(0, |j|^{-a})} | \log |x|| dx \leq C\frac{  \log |j|}{|j|^{2 a}}. 
 \end{equation*}
Note that this first term satisfies the bound in (\ref{eq_le_coef}) as soon as $a\geq 1$. 

To deal with the second integral in (\ref{descomposicion}) we use that $\varphi_j=-\frac{R^2}{\pi^2|j|^2}\Delta \varphi_j$ 
 and Green's formula,
 \begin{eqnarray*} \nonumber
&& \int_{|j|^{-a} \leq |x| \leq R} K_{\alpha \beta}(x)\varphi_{-j}(x)dx\\ \nonumber
&&=-
\frac{R^2}{\pi^2|j|^2}\int_{|j|^{-a} \leq |x| \leq R} K_{\alpha \beta}(x)\Delta \varphi_{-j}(x)dx \\ \nonumber
&& =-\frac{R^2}{\pi^2|j|^2}\int_{|j|^{-a} \leq |x| \leq R} \Delta K_{\alpha \beta}(x) \varphi_{-j}(x)dx   \\ \nonumber
&& \quad + \frac{R^2}{\pi^2|j|^2}\int_{|x|=R} \left(   K_{\alpha \beta}(x)\frac{ \partial \varphi_{-j}}{\partial r}(x) -\varphi_j(x)\frac{\partial  K_{\alpha \beta}}{\partial r}(x)    \right) d \sigma (x)  \\ 
&& \quad  - \frac{R^2}{\pi^2|j|^2}\int_{|x|=|j|^{-a}} \left(   K_{\alpha \beta}(x)\frac{ \partial \varphi_{-j}}{\partial r}(x) -\varphi_j(x)\frac{\partial  K_{\alpha \beta}}{\partial r}(x)    \right) d \sigma (x). 
\end{eqnarray*}
Note that we only have to prove that the three integrals in the right hand side of the second equality above can be bounded by $C \log |j|$. The  second one is obviously bounded since $K_{ \alpha \beta}$ is $C^\infty$  in $0  < |x |  \leq R$. The same argument applies to the first integral in the annulus $\rho  \leq |x |  \leq R$, so that we can reduce ourselves to the subregion $|j|^{-a}<|x|<\rho$. Thus, we only have to show that
\begin{equation}\label{integral_1_2}
\left|  \int_{|j|^{-a} \leq |x| \leq \rho} \Delta \Phi_{\alpha \beta}(x) \varphi_{-j}(x)dx    \right| \leq C \log |j|,
\end{equation}
and 
\begin{equation}\label{integral_2_2}
\left|  \int_{|x|=|j|^{-a}} \left(   \Phi_{\alpha \beta}(x)\frac{ \partial \varphi_{-j}}{\partial r}(x) -\varphi_j(x)\frac{\partial  \Phi_{\alpha \beta}}{\partial r}(x)    \right) d \sigma (x)   \right| \leq C \log |j|,
\end{equation}
where we have used that $K_{\alpha\beta}$ is equal to $\Phi_{\alpha\beta}$ in $B(0,\rho)$.

A straightforward computation shows that
\begin{eqnarray} \nonumber
\Delta \Phi_{\alpha \beta}(x)&=& \Delta \Phi_1(|x|) \delta_{\alpha \beta}+  \Delta \Phi_2 (|x|) \frac{x_\alpha x_\beta}{|x|^2}+(-1)^{\frac{\alpha +\beta}{2}}\Phi_2(|x|)2\frac{x_1^2-x_2^2}{|x|^4}\delta_{\alpha \beta}\\ \label{componentes_2_d}
&& -2\Phi_2(|x|) \frac{x_1x_2}{|x|^4}(1-\delta_{\alpha \beta}),
\end{eqnarray}
and
\begin{equation}\label{derivada_radial_2_d}
\frac{\partial \Phi_{\alpha \beta}}{\partial r}(x)=  \left(\nabla \Phi_1(|x|) \cdot \frac{x}{|x|}\right)\delta_{\alpha \beta}+\left(\nabla \Phi_2(|x|) \cdot \frac{x}{|x|}\right) \frac{x_\alpha x_\beta}{|x|^2}
   ,
   \end{equation}
   where for $n=1,2$ we have
\begin{equation} \label{eq_c1}
	\nabla \Phi_n(|x|)= \left( \frac{1}{\pi |x|} \Psi_n(|x|)+ \frac{\log |x|}{\pi}\Psi'_n(|x|)+\Upsilon'(|x|)   \right) \frac{x}{|x|}, 
	\end{equation}
 \begin{equation} \label{eq_c2}
	\Delta \Phi_n(|x|)=\frac{2+\log |x|}{\pi |x|}\Psi'_n(|x|)+\frac{\log |x| }{\pi} \Psi''_2(|x|)+\frac{1}{|x|}\Upsilon'_n(|x|)+\Upsilon''_n(|x|).
	\end{equation}
   
From (\ref{componentes_2_d}) and (\ref{eq_c2}) we have that 
$$\left| \Delta \Phi_{\alpha \beta}(x)  \right| \leq C \left(  \frac{| \log |x||}{|x|}+\frac{1}{|x|^2}   \right) \leq \frac{C}{|x|^2},  $$ 
and
$$\left|  \int_{|j|^{-a} \leq |x| \leq \rho} \Delta \Phi_{\alpha \beta}(x) \varphi_{-j}(x)dx    \right| \leq C \int_{|j|^{-a}}^\rho \frac{dt}{t} \leq C \log |j|,$$ 
so (\ref{integral_1_2}) holds.

Finally, to prove estimate (\ref{integral_2_2}) we observe that $\frac{\partial \varphi_{-j}}{\partial r}(x)=-\frac{i \pi}{R} (j \cdot x) \varphi_{-j}(x)$, and the following  
$$
\left| \Phi_n(|x|) \right| \leq C (1+ \log |j|), \hspace{0.2cm} \left|       \nabla \Phi_n(|x|)   \right| \leq  C ( 1+  \log |j| + |j|^{a}), \hspace{0.2cm} |x|=|j|^{-a},
$$
for $n=1,2$, 
which are easily obtained from (\ref{eq_c1}) and (\ref{eq_c3}).
Therefore, 
\begin{eqnarray*}
&& \left|  \int_{|x|=|j|^{-a}} \left(   \Phi_{\alpha \beta}(x)\frac{ \partial \varphi_{-j}}{\partial r}(x) -\varphi_j(x)\frac{\partial  \Phi_{\alpha \beta}}{\partial r}(x)    \right) d \sigma (x)   \right| \\
&& \quad \leq C\left( (1+ \log |j|)|j| + ( 1+  \log |j| + |j|^{a}) \right)|j|^{-a}  \\
&& \quad \leq C  \left( 1 + |j|^{1-a}  \log |j|+ |j|^{-a} \log |j| \right).
\end{eqnarray*}
If we take $a =1$, the above  estimate  gives  (\ref{integral_2_2}).

\subsection{The case $d=3$}

The proof is analogous to the case $d=2$ so that we omit some details. The fundamental solution in dimension $d=3$ is given by (see \cite{AK}) 
 \begin{equation*}\label{Matriz_fundamental_3}
\Phi (x)=
\Phi_1(|x|)I+\Phi_2(|x|)J(x) ,
\end{equation*}
and as in the case $d=2$, $J(x)=\frac{x^Tx}{|x|^2}$,
where for $v>0$
\begin{equation*}
\Phi_1(v) =\frac{k_s^2}{4\pi\omega^2}\frac{e^{ik_s v}}{v}
 + \frac{1}{4\pi \omega^2}\left(   \frac{ik_s e^{ik_s v} - ik_pe^{ik_p v}}{v^2} - \frac{e^{ik_s v}-e^{ik_p v}}{v^3}  \right),
\end{equation*}
and
\begin{eqnarray*}
\Phi_2(v)&=&\frac{1}{4\pi \omega^2}   \left(-  \frac{k_s^2e^{ik_sv}-k_p^2e^{ik_pv}}{v} -  3\frac{ik_se^{ik_sv}-ik_pe^{ik_pv}}{v^2}     
\right. \\ && \quad \left. + 3\frac{e^{ik_sv}-e^{ik_pv}}{v^3}        \right)     . 
\end{eqnarray*}

It can be seen that 
\begin{equation}\label{asintotica_1_3} 
 \Phi_1(v)=\frac{k_s^2}{4\pi\omega^2}\frac{e^{ik_sv}}{v}+ \frac{k_p^2-k_s^2}{8\pi \omega^2v}+\Psi_1(v) , 
\end{equation}
\begin{equation}\label{asintotica_2_3}
\Phi_2(v)=\frac{k_s^2-k_p^2}{8\pi \omega^2v}+\Psi_2(v),
\end{equation}
where $\Psi_1$ and $\Psi_2$ are in $ C^\infty [0,\infty )$ and null when $k_p=k_s$. 
  
  From (\ref{asintotica_1_3}) and (\ref{asintotica_2_3})  we have
\begin{eqnarray*}
\Phi_{\alpha \beta}(x)&=&\left(\frac{k_s^2 e^{ik_s|x|}}{4 \pi \omega^2|x|}+ \frac{k_p^2-k_s^2}{8\pi \omega^2|x|}+\Psi_1(|x|) \right) \delta_{\alpha \beta}
 \\
&& +\left( \frac{k_s^2-k_p^2}{8\pi \omega^2|x|}+\Psi_2(|x|)     \right)\frac{x_\alpha x_\beta}{|x|^2}.
\end{eqnarray*}



Again C will indicate a universal constant, depending on the same parameters as in the case $d=2$.  
Let $ j\in \mathbb{Z}^3$, $a>0$  to be chosen later and we assume  that $ |j|^{-a} \leq \min \{ 1/2, \rho \}$. 

\begin{equation}
\widehat{K_{\alpha \beta}}(j)  =\int_{B(0,|j|^{-a})}K_{\alpha \beta}(x)\varphi_{-j}(x)dx \label{descomposicion_3}
 +\int_{|j|^{-a} \leq |x| \leq R}K_{\alpha \beta}(x)\varphi_{-j}(x)dx.
\end{equation}

For the first integral we observe that 
$$K_{\alpha \beta}(x) \leq C \left( 1+ \frac{1}{|x|} \right), \hspace{0.4cm} x \in B(0,R).$$
Therefore, 
\begin{equation}\label{primera_3}
\left|  \int_{|x| \leq |j|^{-a}}K_{\alpha \beta}(x)\varphi_{-j}(x)dx   \right| 
 \leq C \int_{0}^{|j|^{-a}} \left( t +t^2 \right)dt \leq \frac{C}{|j|^{2 a}}.
 \end{equation}
 
 For the second integral in (\ref{descomposicion_3}) we proceed as in the case $d=2$
\begin{eqnarray} \nonumber 
&& \int_{|j|^{-a} \leq |x| \leq R}K_{\alpha \beta}(x)\varphi_{-j}(x)dx\\ \nonumber
&&=-\frac{R^2}{\pi^2|j|^2}\int_{|j|^{-a}  \leq  |x| \leq R}K_{\alpha \beta}(x ) \Delta \varphi_{-j}(x)dx 
\\ \nonumber
&&=-\frac{R^2}{\pi^2|j|^2}\int_{|j|^{-a}  \leq  |x| \leq R} \Delta K_{\alpha \beta}(x )  \varphi_{-j}(x)dx \\ \nonumber
&& \quad +\int_{|x|=R} \left( K_{\alpha \beta}(x) \frac{\partial \varphi_{-j}}{\partial r}(x)- \varphi_{-j}(x) \frac{\partial K_{\alpha \beta}}{\partial r}(x)      \right) d \sigma (x)    \\
&& \quad -\int_{|x|=|j|^{-a}} \left( K_{\alpha \beta}(x) \frac{\partial \varphi_{-j}}{\partial r}(x)- \varphi_{-j}(x) \frac{\partial K_{\alpha \beta}}{\partial r}(x)      \right) d \sigma (x) .  \label{anton} 
\end{eqnarray} 
$K_{\alpha \beta}$ is a $C^\infty$ function  in  $\min\{ 1/2, \rho \} \leq |x| \leq R$, so $\Delta K_{\alpha \beta}$ in bounded in this region. Since $\psi(|x|)=1 $ in $B(0,\rho)$, it will be sufficient to show that
\begin{equation}\label{intefral_13}
\left|   \int_{|j|^{-a}  \leq  |x|  \leq \frac{1}{2} }\Delta \Phi_{\alpha \beta}(x )  \varphi_{-j}(x)dx    \right| \leq C \log |j|,
\end{equation}
and
\begin{equation}\label{integral_23}
\left|    \int_{|x|=|j|^{-a}} \left( \Phi_{\alpha \beta}(x) \frac{\partial \varphi_{-j}}{\partial r}(x)- \varphi_{-j}(x) \frac{\partial \Phi_{\alpha \beta}}{\partial r}(x)      \right) d \sigma (x)  \right| \leq C \log |j|.
\end{equation}

A simple calculation gives us for $|x| \in [|j|^{-a},\frac{1}{2}]$ and  $n=1,2$  that
$$
 |\Phi_n(|x|)| \leq \frac{C}{|x|},\quad   |\nabla \Phi_n(|x|)|\leq \frac{C}{|x|^2}, \quad    |\frac{\partial \Phi_n}{\partial r}(|x|)|, \leq \frac{C}{|x|^2},\quad   |\Delta \Phi_n(|x|)|\leq \frac{C}{|x|}.
$$
Moreover, if $g_{\alpha \beta}(x)=\frac{x_\alpha x_\beta}{|x|^2}$ then $|\nabla g_{\alpha   \beta}(x)| \leq \frac{C}{|x|}$ and   $|\Delta g_{\alpha \beta}(x)| \leq \frac{C}{|x|^2}$. 

From these estimates we obtain
\begin{equation*}\label{estimaciones_3_d}
\left| \Phi_{\alpha \beta } \right| \leq \frac{C}{|x|}, \hspace{0.15cm} \left| \nabla \Phi_{\alpha \beta}(x)   \right|\leq \frac{C}{|x|^2}, \hspace{0.15cm} \left| \Delta \Phi_{\alpha \beta}(x)   \right|\leq \frac{C}{|x|^3} , \hspace{0.15cm}  \left|   \frac{\partial \Phi_{\alpha \beta}}{\partial r}(x)  \right|\leq \frac{C}{|x|^2}.
\end{equation*}
We start by estimating (\ref{intefral_13}) by using  the above estimates

\begin{equation*}\label{integra_131}
\left|   \int_{|j|^{-a}  \leq  |x|  \leq \frac{1}{2}} \Delta \Phi_{\alpha \beta}(x )  \varphi_{-j}(x)dx    \right|    \leq C \int_{|j|^{-a}}^{1/2}  \frac{1}{t}d t  \leq C \log |j|.
\end{equation*}

To study (\ref{integral_23}) we use the same estimates we used to obtain (\ref{intefral_13}) and 
$\frac{\partial \varphi_{-j}}{\partial r}(x)=-i \frac{\pi}{R}\frac{j \cdot x}{|x|} \varphi_{-j}(x).$
We have
\begin{equation*}\label{integral_231}
\left|    \int_{|x|=|j|^{-a}}  \Phi_{\alpha \beta}(x) \frac{\partial \varphi_{-j}}{\partial r}(x) d \sigma (x) \right|   \leq C |j|^{1+a}\int_{|x|=|j|^{-a}}d\sigma (x) \leq C |j|^{1-a},
\end{equation*}
and
\begin{equation*}\label{integral_232}
\left|    \int_{|x|=|j|^{-a}} \varphi_{-j}(x) \frac{\partial \Phi_{\alpha \beta}}{\partial r}(x)     d \sigma (x)  \right| \leq C |j|^{2 a}   \int_{|x|=|j|^{-a}} d \sigma (x)\leq C.
\end{equation*}
If we take $a =1$, from  the inequalities above, (\ref{descomposicion_3}), (\ref{primera_3}), (\ref{anton}), (\ref{intefral_13}) and (\ref{integral_23}) we obtain estimate (\ref{eq_le_coef}).

\section{Implementation and numerical experiments}

The discretization of the Lippmann-Schwinger equation (\ref{eq_LSv}) reduces to the finite dimensional implicit system (\ref{eq_LSv3m}). The solution requires,
\begin{enumerate}
\item An approximation of the Fourier coefficients $\hat K_h$.
\item A solver for the linear system. Here we use the gmres routine in MATLAB. 
\end{enumerate} 
To approximate the Fourier coefficients $\hat K_{\alpha \beta}$ we consider  the trapezoidal rule. This can be interpreted as  the discrete Fourier transform of  $K_{\alpha \beta}$ in a suitable uniform mesh centered in the origin. Therefore we use the fft algorithm in MATLAB. 

We must be careful since the Green function is singular at the origin. This affects to the numerical approximation of the Fourier coefficients in (\ref{eq_foco}), where the Green tensor appears in the integrand. In particular, if we use the trapezoidal rule to approximate the integral, we obtain a singular value for this integral. To avoid this problem we simply replace the value at the origin by zero. This does not affect significantly to the accuracy of the trapezoidal rule. For example, in dimension $d=2$ the functions in the Green tensor have a logarithmic singularity at the origin.  The error derived by our choice affects the quadrature formula only near $ x=0$. More precisely, it affects the ball centered at $ x=0$ and radius $h$ (the mesh size) $B(0,h)$. The real value of the integral in this ball is easily estimated by
$$
\int_{B(0,h)} \log |x| dx =2\pi \int_0^h r \log(r) \; dr \sim h^2 \log h.
$$ 
Thus, if we replace the integral by the value obtained with the trapezoidal rule and our choice, the error is basically the same as the one associated to the trapezoidal rule for smooth functions.

Now we show some numerical experiments that illustrate the convergence of the numerical method in dimension $d=2$ and provide numerical approximations of scattering data.

\bigskip

{\bf Experiment 1} In this first experiment we illustrate the convergence of the solution of the Lippmann-Schwinger equation by considering an analytical solution. We define the vector function 
$$
{\bf g} (x)=\chi_{|x|<1}(1-|x|^2)^4(1, 1).  
$$
Then, ${\bf v}=(\Delta^*+\omega^2I){\bf g}$, ${\bf f}={\bf v}-{\bf g}$ and $Q=I$ solve the Lippmann-Schwinger equation (\ref{eq_LSv}). In table \ref{tab1} we illustrate the convergence of the solution as $N$ grows. 

\begin{table}
\begin{center}
\begin{tabular}{|c|c|c|}
h & $\| {\bf v}_h-{\bf v} \|_{L^\infty}$ & $\| {\bf v}_h-{\bf v} \|_{L^2}$\\ \hline
$2^{-4}$ & $7.1521 \times 10^{-2}$ & $5.1043 \times 10^{-2}$\\
$2^{-6}$ & $5.6033 \times 10^{-3}$ & $4.0045 \times 10^{-3}$\\
$2^{-8}$ & $4.2506 \times 10^{-4}$ & $3.0240 \times 10^{-4}$
\end{tabular}
\caption{Experiment 1: error estimate as $h$ decreases. \label{tab1}}
\end{center}
\end{table}

\bigskip

{\bf Experiment 2}. Here we show an example of the scattering fields and amplitudes that can be obtained with the numerical method presented in this paper. Unfortunately we do not have analytical solutions in this case to compare the approximation. 

We consider the Lamé parameters $\lambda=1$, $\mu=4$ and the potential given by 
$$
Q(x)=q(x)\; I, \quad q(x)=\chi_{0.6<|x|<0.8}+1.2 \chi_{|x_1|+|x_2|<0.2}
$$
which has compact support in $B(0,1)$. 

We fix the angle of an incident transverse wave $\theta=(\cos\pi/4, \sin \pi/4)$. In Figure \ref{fig2} we show the real part of  the transverse wave $Re\; ({\bf v}_s(x)\cdot \theta^\perp)$ (here $\theta^\perp=(\cos 3\pi/4, \sin 3\pi/4)$) for $\omega=50$ both in the computational domain $[-2.1,2.1]\times [-2.1,2.1]$ and the domain where the approximation converges to the solution of the real continuous equation $[-1,1]\times [-1,1]$. In the computational domain the solution is periodic, due to the trigonometric basis that we use. Of course,  this is not the case for the inner subdomain where we approximate the real solution. 

\begin{figure}%
\begin{tabular}{cc}
\includegraphics[width=6cm]{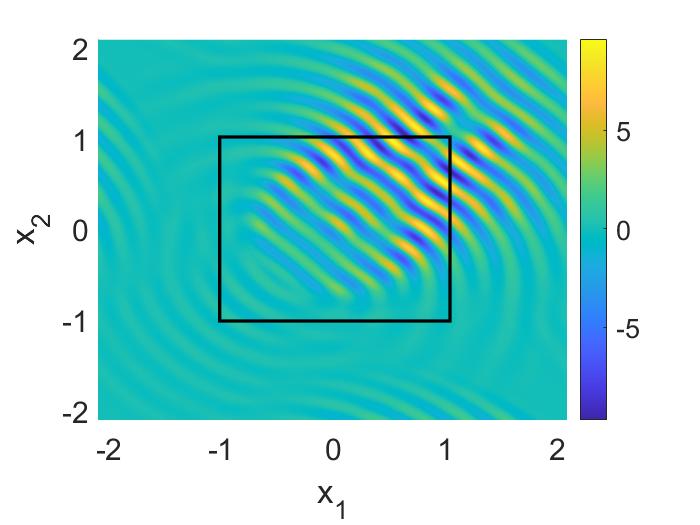}%
& \includegraphics[width=6cm]{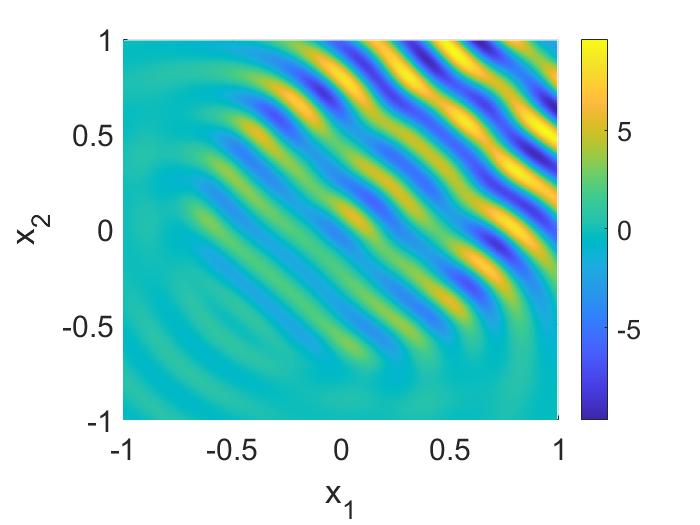}\\
Computational domain & Real approximation domain, 
\end{tabular}
\caption{Numerical results of experiment 2: $Re\; ({\bf v}_s( x)\cdot \theta^\perp)$   corresponding to a transverse incident wave with incident angle $\theta =(\cos\pi/4, \sin \pi/4)$ and $\omega=50$, both in the computational domain (left) and the subdomain where the approximation to the continuous solution holds (right).  }%
\label{fig2}
\end{figure}

Now  we compare both the real parts of the longitudinal $Re\; ({\bf v}_s(x)\cdot \theta)$ and transverse $ Re (\;{\bf v}_s(x)\cdot \theta^\perp)$ solutions for transverse incident waves at different frequencies and the same angle $\theta=(\cos\pi/4, \sin \pi/4)$. Note that here the incident wave $u^s_i= e^{ik_s \theta\cdot x}\theta^\perp$ is complex and so it is the scattered field. The real part (resp. imaginary part) corresponds to the real part (resp. imaginary part) of the incident wave.
In particular, we observe how larger incident frequencies provide more complicated solutions. A similar comparison for the longitudinal incident wave  $u^s_i= e^{ik_s \theta\cdot x}\theta$ is given in Figure \ref{fig4}.

\begin{figure}%
\begin{tabular}{cc}
\includegraphics[width=6cm]{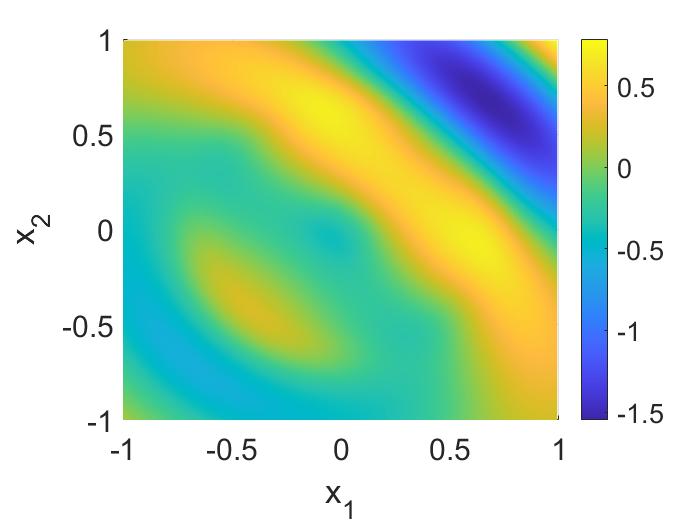}%
& \includegraphics[width=6cm]{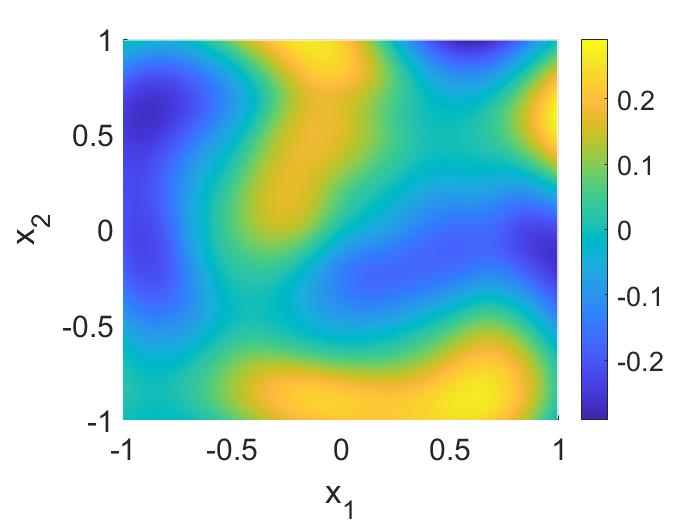}\\
$Re\; ({\bf v}_s\cdot \theta^\perp)$, $\omega=10$ & $Re\; ({\bf v}_s\cdot \theta)$, $\omega=10$, \\
\includegraphics[width=6cm]{u_s_k50_real_trans_real.jpg}%
& \includegraphics[width=6cm]{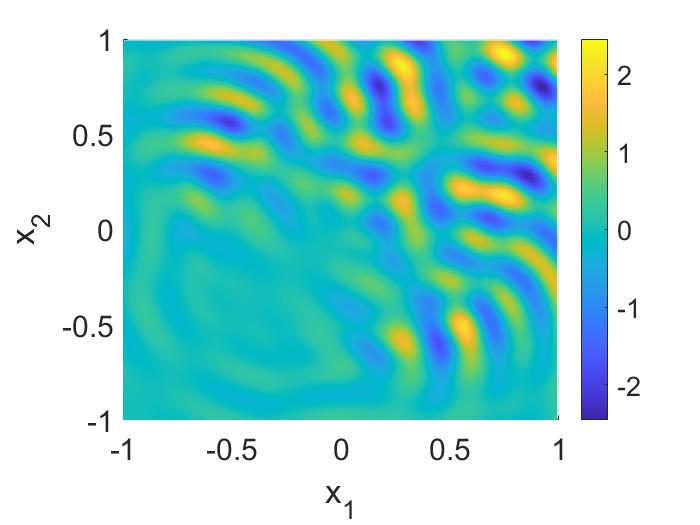}\\
$Re\; ({\bf v}_s\cdot \theta^\perp)$, $\omega=50$ & $Re\; ({\bf v}_s\cdot \theta)$, $\omega=50$, \\
\includegraphics[width=6cm]{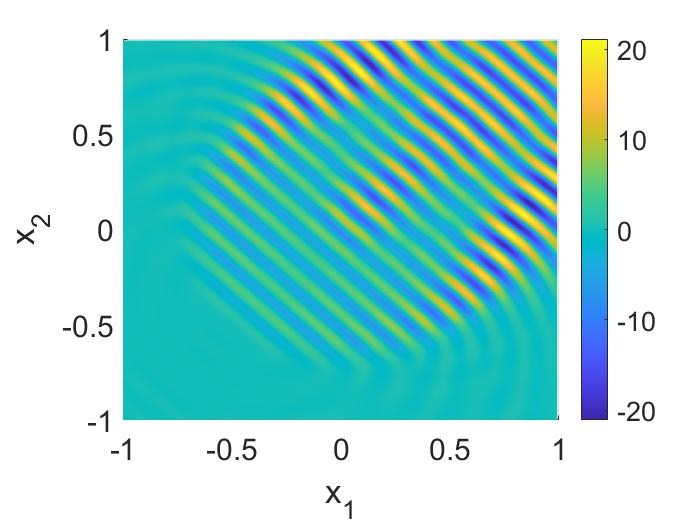}%
& \includegraphics[width=6cm]{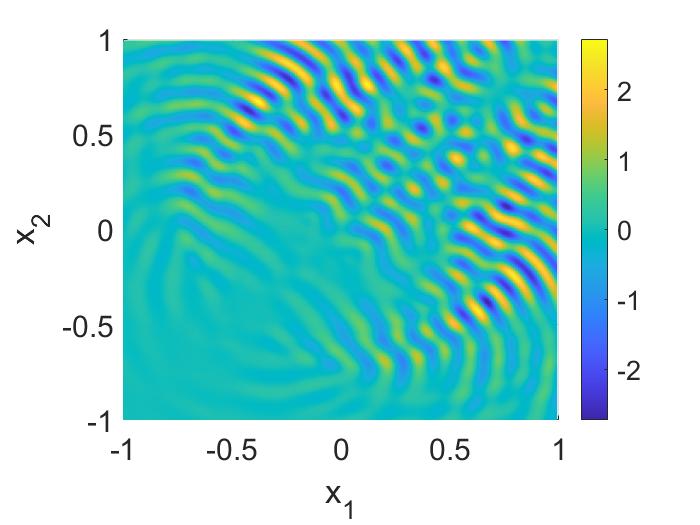}\\
$Re\; ({\bf v}_s\cdot \theta^\perp)$, $\omega=100$ & $Re\; ({\bf v}_s\cdot \theta)$, $\omega=100$ .
\end{tabular}
\caption{Numerical results of experiment 2: We show transverse $Re\; ({\bf v}_s( x)\cdot \theta^\perp)$ (left) and  longitudinal $Re\; ({\bf v}_s(x)\cdot \theta)$ (right) waves corresponding to a transverse incident wave with incident angle $\theta =(\cos \pi/4, \sin \pi/4)$ and different frequencies $\omega=10,50$ and $100$.  }%
\label{fig3}
\end{figure}

\begin{figure}%
\begin{tabular}{cc}
\includegraphics[width=6cm]{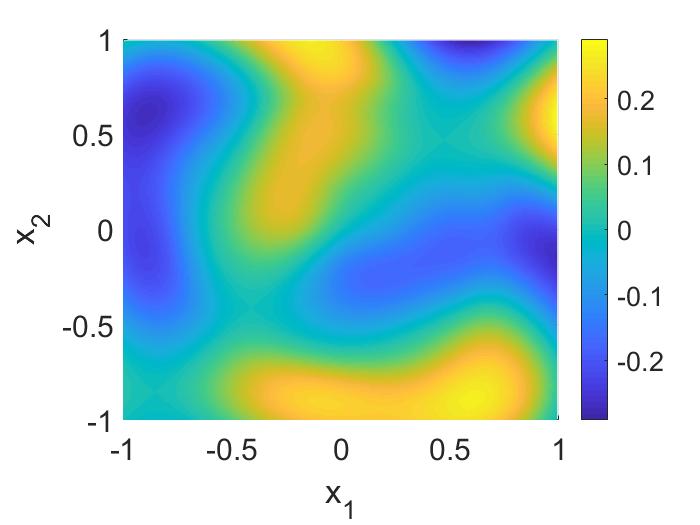}%
& \includegraphics[width=6cm]{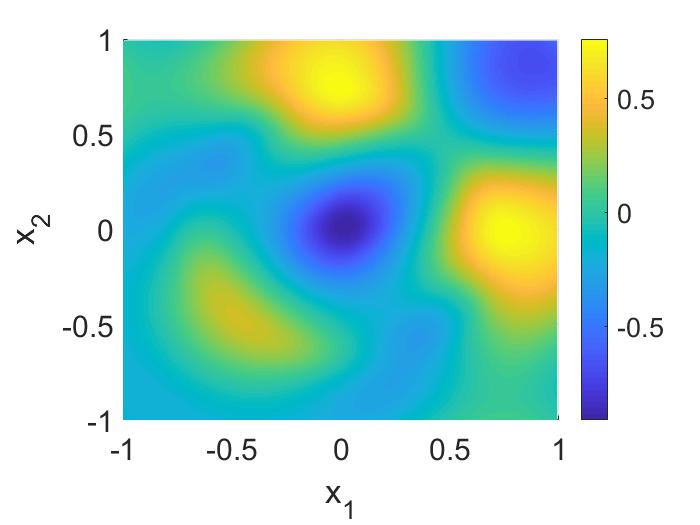}\\
$Re\; ({\bf v}_p\cdot \theta^\perp)$, $\omega=10$ & $Re\; ({\bf v}_p\cdot \theta)$, $\omega=10$, \\
\includegraphics[width=6cm]{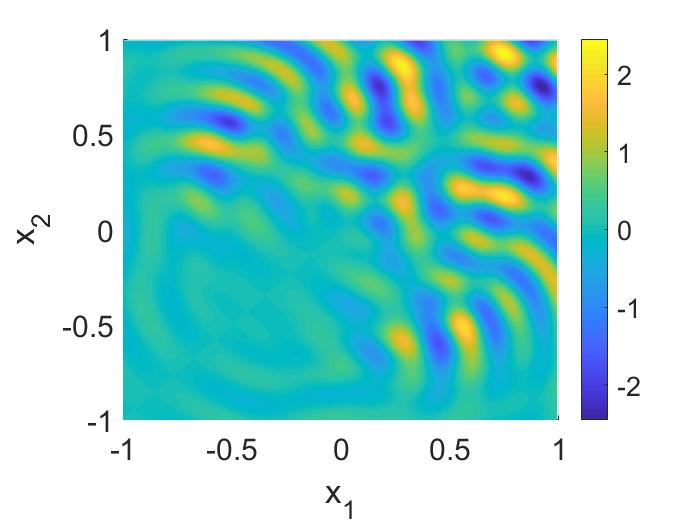}%
& \includegraphics[width=6cm]{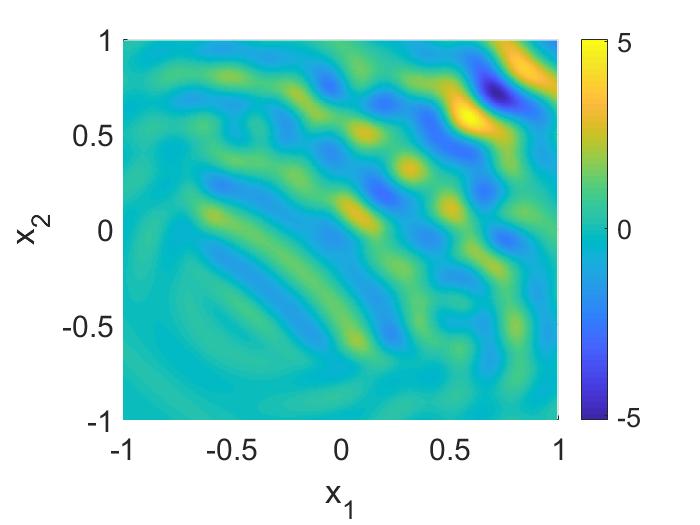}\\
$Re\; ({\bf v}_p\cdot \theta^\perp)$, $\omega=50$ & $Re\; ({\bf v}_p\cdot \theta)$, $\omega=50$, \\
\includegraphics[width=6cm]{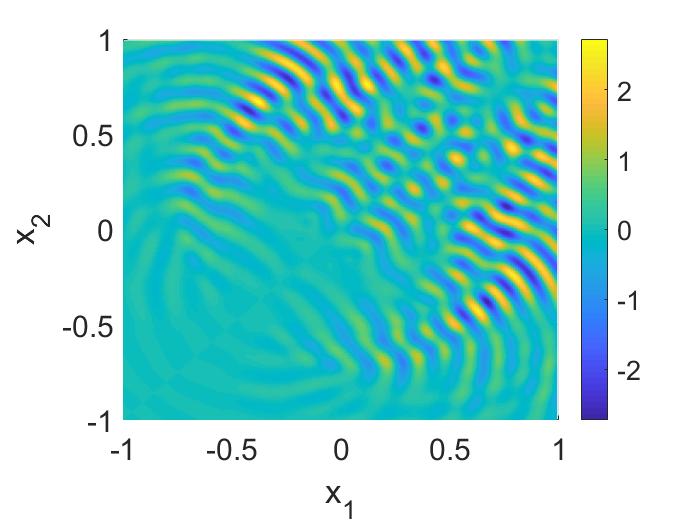}%
& \includegraphics[width=6cm]{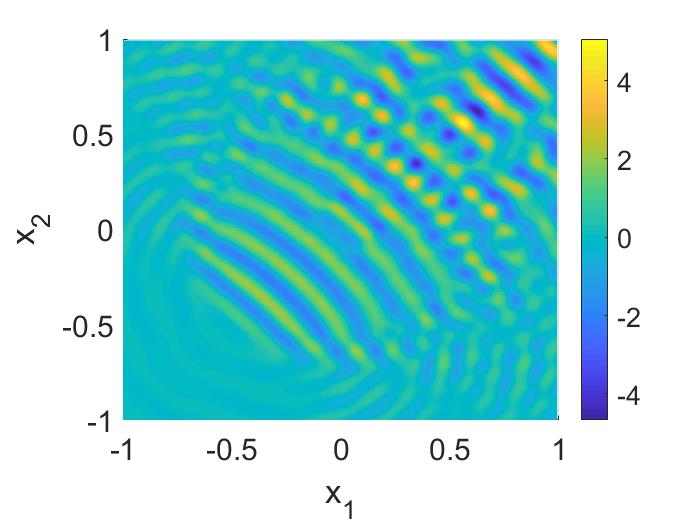}\\
$Re\; ({\bf v}_p\cdot \theta^\perp)$, $\omega=100$ & $Re\; ({\bf v}_p\cdot \theta)$, $\omega=100$ .
\end{tabular}
\caption{Numerical results of experiment 2: We show transverse $Re\; ({\bf v}_s( x)\cdot \theta^\perp)$ (left) and longitudinal $Re\; ({\bf v}_s(x)\cdot \theta)$ (right) waves corresponding to a longitudinal incident wave with incident angle $\theta =(\cos \pi/4, \sin \pi/4)$ and different frequencies $\omega=10,50$ and $100$.  }%
\label{fig4}
\end{figure}



Finally, in Figure \ref{fig1} the sinogram with the scattering amplitudes corresponding to the real part of the transverse waves $Re\; ({\bf v}^s_s(\omega,\theta,\theta'))$ and longitudinal ones $Re \;({\bf v}^s_p(\omega,\theta,\theta'))$, for a traverse incident wave ${\bf u_i^s}=e^{ik_s \theta} \theta^\perp$, with $\theta=(\cos\pi/4, \sin \pi/4)$, different frequencies $\omega>0$ and angles $\theta'\in \mathbb{S}^1$ (upper simulations). In the lower simulations, we show the analogous sinograms corresponding to a longitudinal incident wave ${\bf u_i^p}=e^{ik_p \theta} \theta$, with $\theta=(\cos\pi/4, \sin \pi/4)$.  

\begin{figure}%
\begin{tabular}{cc}
\includegraphics[width=6cm]{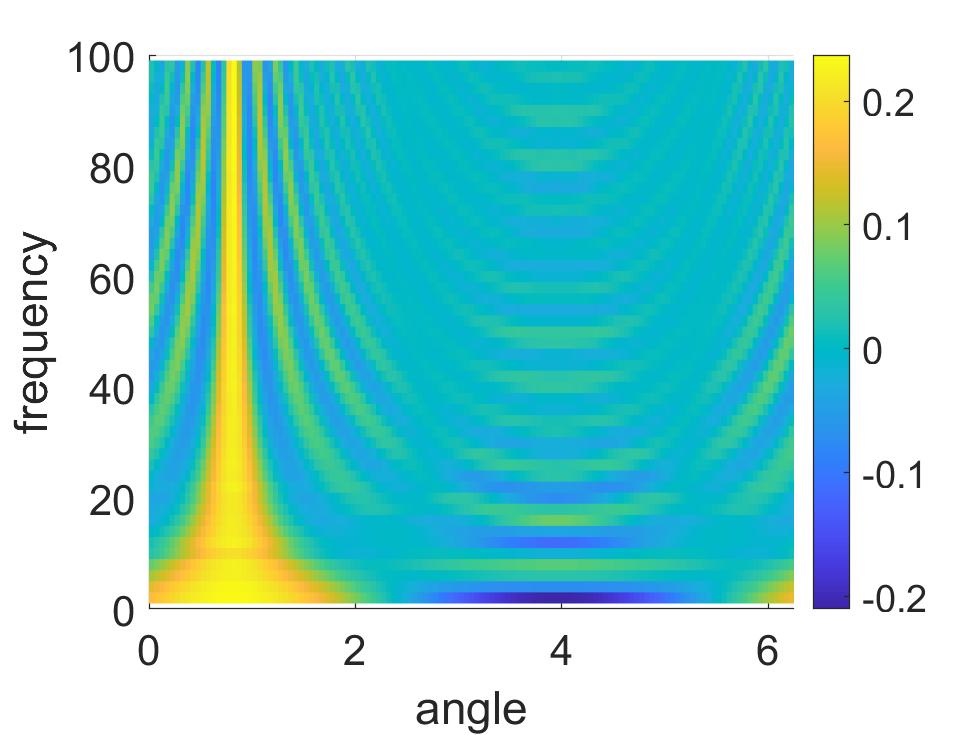}%
& \includegraphics[width=6cm]{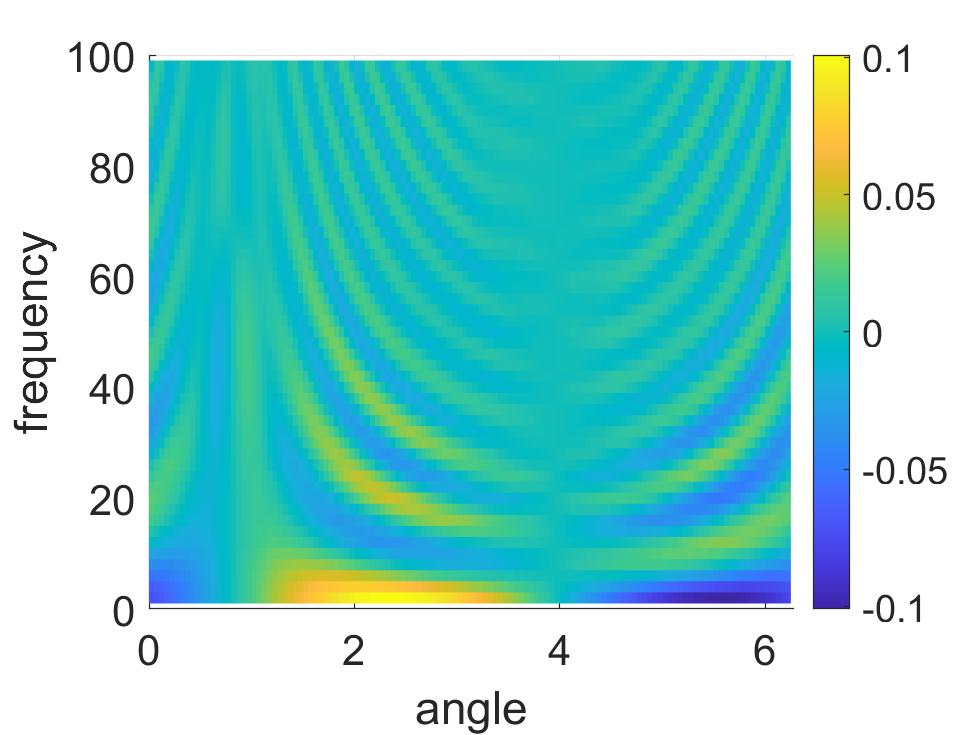}\\
Transverse amplitude & Longitudinal amplitude, \\
\includegraphics[width=6cm]{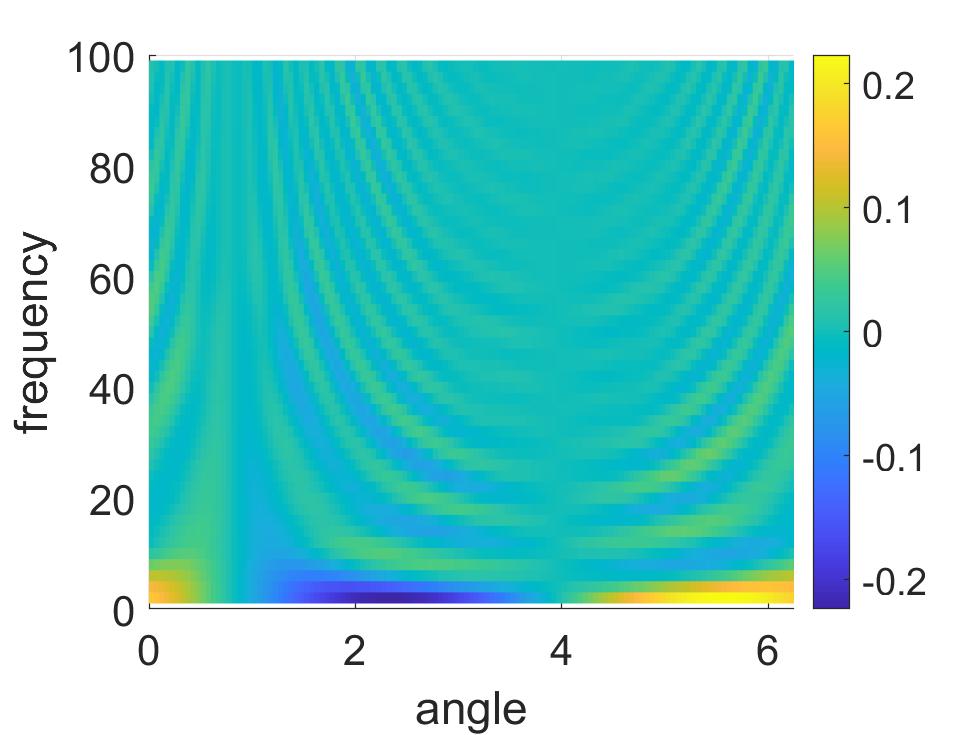}%
& \includegraphics[width=6cm]{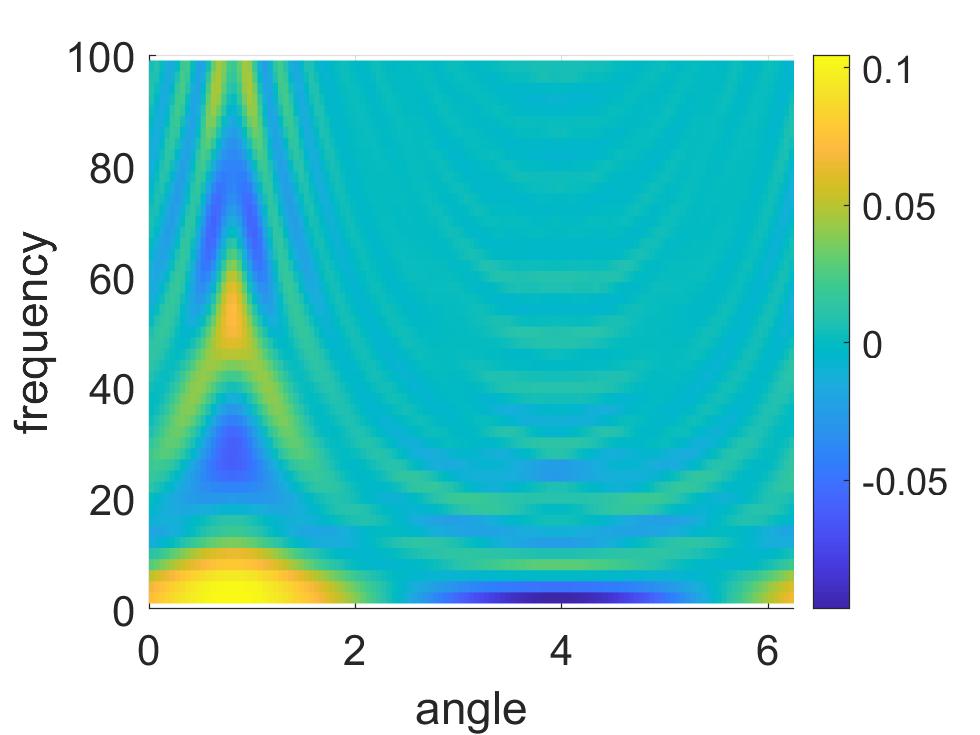}\\
Transverse amplitude & Longitudinal amplitude,
\end{tabular}
\caption{Numerical results of experiment 2: scattering amplitudes corresponding to a transverse incident wave with incident angle $\theta =(\cos\pi/4, \sin \pi/4)$ (upper simulations), and longitudinal incident wave with the same incident angle $\theta =(\cos\pi/4, \sin \pi/4)$ (lower simulations). The horizontal axis corresponds to the angle $\theta_a'\in[0,2\pi)$ such that $\theta'=(\cos \theta_a',\sin \theta_a')\in \mathbb{S}^1$ and the vertical axis is the frequency $\omega\in(0,100)$.}%
\label{fig1}
\end{figure}

When considering a transverse incident wave (upper simulations in figure \ref{fig1}) we observe that transverse amplitudes are more concentrated for $\theta'=\theta=(\cos\pi/4, \sin \pi/4)$ and larger $\omega$ while longitudinal ones have smaller amplitude for $\theta'=\theta=(\cos\pi/4, \sin \pi/4)$ and  $\theta'=\theta^\perp=(\cos 3\pi/4, \sin 3\pi/4)$. On the other hand, for a longitudinal  incident wave (lower simulations in figure \ref{fig1}) we observe the opposite phenomenon, i.e. transverse amplitudes have  smaller amplitude for $\theta'=\theta=(\cos\pi/4, \sin \pi/4)$ and  $\theta'=\theta^\perp=(\cos 3\pi/4, \sin 3\pi/4)$, while longitudinal ones are more concentrated for $\theta'=\theta=(\cos\pi/4, \sin \pi/4)$ and larger $\omega$.

\section*{Acknowledgements} The authors acknowledge the support of Ministerio de Ciencia, Innovaci\'on y Universidades of
the Spanish goverment through grant MTM2017-85934-C3-3-P

\end{document}